\documentclass[11pt,a4paper,reqno]{article}
\usepackage{graphicx}
\usepackage[margin=3cm,includehead]{geometry}

\usepackage{amssymb}
\usepackage{amsmath}
\usepackage{latexsym}
\usepackage{amscd}
\usepackage{amsthm}
\usepackage{mathrsfs}
\usepackage{url}
\usepackage{amsmath,mathtools}
\usepackage[T1]{fontenc}
\usepackage{tikz-cd}
\usepackage{titlesec}
 \usepackage[titletoc,toc,title]{appendix}
\graphicspath{}
\usepackage{fancyhdr} 

\newcommand\shorttitle{Deformations of calibrated subbundles} 
\newcommand\authors{Romy Marie Merkel} 

\titleformat{\section}  
    {\normalfont\large\bfseries\center}{\thesection.}{1em}{}
    
\makeatletter
\newcommand*{\rom}[1]{\expandafter\@slowromancap\romannumeral #1@}
\makeatother

\usepackage{thmtools}
\newtheorem{satz}{Theorem}[section]
\newtheorem{cor}[satz]{Corollary}
\newtheorem{lemma}[satz]{Lemma}

\theoremstyle{definition}
\newtheorem{defi}[satz]{Definition}

\newtheorem{rem}[satz]{Remark}
\newtheorem*{ack}{Acknowledgments}
\numberwithin{equation}{section}
\usepackage[english]{babel}
\usepackage{enumitem}

\usepackage[
style=alphabetic,
natbib=true, 
hyperref=true,
backend=biber
]{biblatex} 

\bibliography{master.bib}

\usepackage{amsmath}
\usepackage{upgreek}
\usepackage{commath}
\usepackage{amssymb}
\usepackage{amsthm}
\usepackage{csquotes}
\allowdisplaybreaks

\usepackage{graphicx,tikz,tabularx} 
\usepackage[draft=false,babel,tracking=true,kerning=true,spacing=true]{microtype} 

\usetikzlibrary{calc}
\newcommand*{\vertchar}[2][0pt]{
	\tikz[
	inner sep=0pt,
	shorten >=-.15ex,
	shorten <=-.15ex,
	line cap=round,
	baseline=(c.base),
	]\draw
	(0,0) node (c) {#2}
	($(c.south)+(#1-0.5,0)$) -- ($(c.north)+(#1+1.5,0)$);
}

\usetikzlibrary{cd}

\DeclareMathAlphabet{\mathbbmsl}{U}{bbm}{m}{sl}

\newcommand\spinS{\vertchar{$\mathcal{S}$}}
\newcommand\negspinS{\vertchar{$\mathcal{S}$}_{\!-}}
\newcommand\posspinS{\vertchar{$\mathcal{S}$}_{\!+}}
\newcommand\pmspinS{\vertchar{$\mathcal{S}$}_{\!\pm}}
\newcommand\e{\mathbf{e}}

\newcommand\spn{\text{\normalfont span}}
\newcommand\vol{\text{\normalfont vol}}
\newcommand\im{\text{\normalfont Im}}
\newcommand\re{\text{\normalfont Re}}
\newcommand\Ver{\text{\normalfont Vert}}
\newcommand\Hor{\text{\normalfont Hor}}
\newcommand\Cl{\text{\normalfont Cl}}

\newcommand\Spin{\text{\normalfont Spin}}
\newcommand\End{\text{\normalfont End}}

\newcommand\Orm{\text{\normalfont O}}

\newcommand\U{\text{\normalfont U}}
\newcommand\SU{\text{\normalfont SU}}
\newcommand\R{\mathbb{R}}
\newcommand\C{\mathbb{C}}
\newcommand\Hbb{\mathbb{H}}
\newcommand\Oct{\mathbb{O}}
\newcommand\Gtwo{\text{\normalfont G}_2}
\newcommand\Tr{\text{Tr}\,}
\newcommand\dimL{q}
\newcommand\pointM{p}
\newcommand\pointL{p}

\newcommand\Thm{Thm.~}
\newcommand\Ch{Ch.~}
\newcommand\Sec{Sec.~}
\newcommand\Subsec{\Sec}

\newcommand\Prop{Prop.~}

\newcommand\eg{e.g., }
\newcommand\cf{cf.~}

\DeclareMathSymbol{\shortminus}{\mathbin}{AMSa}{"39}

\usepackage{hhline}

\usepackage[svgnames, x11names]{xcolor}
\usepackage{hyperref}
\hypersetup{
	linktocpage=true,
	colorlinks = true,
	linkcolor = Red2,
	anchorcolor = black,
	citecolor = [RGB]{24, 40, 219},
	filecolor = black,
	urlcolor = black,
	hypertexnames = false
}

\colorlet{RED2}{Red2}

\AtBeginDocument{%
}

\begin{document}

\title{Deformations of calibrated subbundles in noncompact manifolds of special holonomy via twisting by special sections}
\author{Romy Marie Merkel}
\date{\today}
\maketitle

\begin{abstract}
	\noindent We study special Lagrangian submanifolds in the Calabi--Yau manifold $T^*S^n$ with the Stenzel metric, as well as calibrated submanifolds in the $\Gtwo$-manifold $\Lambda^2_-(T^*X)$ \smash{$(X^4 = S^4, \mathbb{CP}^2)$} and the $\Spin(7)$-manifold $\negspinS(S^4)$, both equipped with the Bryant--Salamon metrics. 
	We twist naturally defined calibrated subbundles by sections of the complementary bundles and derive conditions for the deformations to be calibrated. 
	We find that twisting the conormal bundle $N^*L$ of $L^{\dimL} \subset S^n$ by a $1$-form $\mu \in \Omega^1(L)$ does not provide any new examples because the Lagrangian condition requires $\mu$ to vanish. 
	Furthermore, we prove that the twisted bundles in the $\Gtwo$- and $\Spin(7)$-manifolds are associative (coassociative) and Cayley, respectively, if the base is minimal (negative superminimal) and the section holomorphic (parallel). 
	This demonstrates that the (co\nobreakdash-)associative and Cayley subbundles allow deformations destroying the linear structure of the fiber, while the base space remains of the same type after twisting.
	While the results for the two spaces of exceptional holonomy are in line with the findings in Euclidean spaces established by \citeauthor{CaliA2} (\citeyear{CaliA2}), the special Lagrangian bundle construction in $T^*S^n$ is much more rigid than in the case of $T^*\R^n$. 
\end{abstract}

\tableofcontents

\section{Introduction}

The notion of calibrations and calibrated submanifolds originates from the seminal paper \cite{HarveyLawson1982} of \citeauthor{HarveyLawson1982}.
Apart from the rich theory of calibrated submanifolds, the link between calibrated geometry and gauge theory (see, \eg \cite{Lotay} for some examples) has been the reason for a lot of work on special Lagrangian submanifolds in Calabi--Yau manifolds, as well as some on (co\nobreakdash-)associative and Cayley submanifolds in $\Gtwo$- and $\Spin(7)$-manifolds, respectively (see, \eg \cite{Lotay2005, Lotay2006a, Lotay2006, Lotay2012} and the references therein). 
The core of this paper lies in constructing special Lagrangian submanifolds in the Calabi--Yau manifold $T^*S^n$ with the Stenzel metric, as well as calibrated submanifolds in the $\Gtwo$-manifold $\Lambda^2_-(T^*X)$ \smash{$(X^4 = S^4, \mathbb{CP}^2)$} and the $\Spin(7)$-manifold $\negspinS(S^4)$, both equipped with the Bryant--Salamon metrics. 

The presented constructions are motivated by the works of Ionel--Ka\-ri\-gian\-nis--Min-Oo \cite{CaliA}, Karigiannis--Leung \cite{CaliA2} and Karigiannis--Min-Oo \cite{CaliB}.
Inspired by the Harvey--Lawson bundle construction of special Lagrangian submanifolds in $\C^n$ \cite{HarveyLawson1982}, Ionel--Karigiannis--Min-Oo \cite{CaliA} described similar constructions of (co\nobreakdash-)associative submanifolds in $\R^7$ and Cayley submanifolds in $\R^8$. 
The idea is to view the ambient manifold as the total space of a vector bundle over some Euclidean space $\R^n$, restricting it to an oriented immersed submanifold $L \subset \R^n$ and then considering the total spaces of appropriate subbundles. 
More precisely, Harvey--Lawson \cite{HarveyLawson1982} viewed $\C^n \cong T^*\R^n$ as the cotangent bundle and then considered the conormal bundle $N^*L^{\dimL}$. 
Similarly, Ionel--Karigiannis--Min-Oo \cite{CaliA} viewed $\R^7 \cong\Lambda^2_-(T^*\R^4)$ as the space of anti-self-dual $2$-forms on $\R^4$ and $\R^8 \cong \negspinS(\R^4)$ as the negative spinor bundle of $\R^4$. 
They examined naturally defined subbundles $E$ and $F = E^\perp$ of $\Lambda^2_-(T^*\R^4)\vert_{L^2}$ of rank $1$ and $2$, and $V_+$ and $V_- = V_+^\perp$ of $\negspinS(\R^4)\vert_{L^2}$ of rank $2$. 
Later, Karigiannis--Min-Oo \cite{CaliB} generalized these constructions to complete, nonflat, noncompact manifolds of special holonomy which are total spaces of vector bundles over a compact base. In other words, they examined the analogs of these submanifolds in the Calabi--Yau manifold $T^*S^n$ with the Stenzel metric, in \smash{$\Lambda^2_-(T^*X)$ ($X^4= S^4, \mathbb{CP}^2$)} with the Bryant--Salamon metric of holonomy $\Gtwo$ and in $\negspinS(S^4)$ with the Bryant--Salamon metric of holonomy $\Spin(7)$. 
The authors of \cite{HarveyLawson1982}, \cite{CaliA} and \cite{CaliB} proved the following results: 
First, the conormal bundle $N^*L$ is special Lagrangian in $T^*X$ if and only if $L^{\dimL}$ is austere in $X^n = \R^n, S^n$. 
Second, the submanifold $E$ ($F$) is associative (coassociative) in \smash{$\Lambda^2_-(T^*X)$} if and only if $L^2$ is minimal (negative superminimal) in $X^4=\R^4, S^4, \mathbb{CP}^2$. 
Third, the submanifold $V_\pm$ is Cayley in $\negspinS(X)$ if and only if $L^2$ is minimal in $X^4=\R^4, S^4$. 

Inspired by Borisenko \cite{Borisenko1993}, Karigiannis--Leung \cite{CaliA2} further generalized \cite{CaliA} by \enquote{twisting} the subbundles by special sections of the complementary bundles. 
They derived conditions on $L$ and the sections in order to obtain calibrated submanifolds of the Euclidean spaces \smash{$\C^n \cong T^*\R^n$}, \smash{$\R^7 \cong\Lambda^2_-(T^*\R^4)$} and $\R^8 \cong \negspinS(\R^4)$. 
The present paper can be seen as a generalization of \cite{CaliA2} to the case of complete, nonflat, noncompact manifolds.

To begin with, we recall the four main examples of calibrated geometries in \autoref{sec:calibratedgeom}. This includes a collection of different characterizations of (co\nobreakdash-)associative and Cayley submanifolds established in \cite{HarveyLawson1982, CaliB}.
\autoref{sec:2ndfundform} introduces the second fundamental form and gives two identities which are key in the presented proofs. 

\autoref{sec:sLagNEW}, \autoref{sec:coassNEW} and \autoref{sec:CayleyNew} represent the core of this paper: We generalize the constructions in \cite{CaliA2} to complete, nonflat, noncompact manifolds of special holonomy. In other words, we twist the calibrated subbundles in $T^*S^n$, $\Lambda^2_-(T^*X)$ ($X^4= S^4, \mathbb{CP}^2$) and $\negspinS(S^4)$ constructed in \cite{CaliB} by special sections. Our main results are contained in \autoref{prop:Lag}, \autoref{thm:assNEU} and \autoref{thm:cayleyNEU}.
In \autoref{prop:Lag}, we find that twisting the conormal bundle $N^*L$ by a $1$-form $\mu \in \Omega^1(L^{\dimL})$ does not provide any new examples because the Lagrangian condition requires $\mu$ to vanish. 
This differs from the case of $\R^4$ in \cite{CaliA2}, where the authors found that the twisted conormal bundle is special Lagrangian in $T^*\R^n$ if and only if $\mu$ is closed and its symmetrized covariant derivative satisfies certain equations. 
\autoref{thm:assNEU} describes the (co\nobreakdash-)associative case: We show that the bundle $E$ twisted by a section $\sigma \in \Gamma(F)$ is associative in $\Lambda^2_-(T^*X)$ if and only if $L^2$ is minimal in $X^4 = S^4, \mathbb{CP}^2$ and $\sigma$ is holomorphic.
On the other hand, the complementary bundle $F$ twisted by a section $\eta \in \Gamma(E)$ is coassociative if and only if $L^2$ is negative superminimal in $X^4$ and $\eta$ is parallel.
Lastly, \autoref{thm:cayleyNEU} proves that the bundle $V_+$ twisted by a section $\psi \in \Gamma(V_-)$ is Cayley in $\negspinS(S^4)$ if and only if $L^2$ is minimal in $S^4$ and $\psi$ is holomorphic. 
The conditions on $L^2$ and the sections $\sigma \in \Gamma(F),\, \eta \in \Gamma(E)$ and $\psi \in \Gamma(V_-)$ for the $\Gtwo$-manifold $\Lambda^2_-(T^*X)$ \smash{$(X^4 = S^4, \mathbb{CP}^2)$} and the $\Spin(7)$-manifold $\negspinS(S^4)$ turn out to be the same as in the case of $\R^4$ in \cite{CaliA2}.
However, in contrast to \cite{CaliA2}, none of the presented proofs rely on identifications with the (purely imaginary) octonions. Instead, they are based on the vanishing of certain (bundle-valued) differential forms, as established in \cite{HarveyLawson1982, CaliB}. 

In \autoref{sec:examples}, we use \autoref{thm:assNEU} to construct explicit examples of associative and coassociative submanifolds in $\Lambda^2_-(T^*S^4)$ for $L^2$ being the equatorial sphere and the Veronese immersion of $S^2(\sqrt{3})$. 
In both cases, we find new coassociative submanifolds, but only the Veronese immersion admits nontrivial holomorphic sections of $F$ and thereby leads to new associative examples.

Our findings demonstrate that the constructions of calibrated submanifolds in Euclidean spaces in \cite{CaliA2} cannot be entirely extended to the manifolds $T^*S^n,\, \Lambda^2_-(T^*X)$ ($X^4 = S^4, \mathbb{CP}^2$) and $\negspinS(S^4)$ considered in \cite{CaliB}.
While the results for the two spaces of exceptional holonomy are in line with the previous findings, the construction in $T^*S^n$ does not provide any new examples.
As in \cite{CaliA2}, the (co\nobreakdash-)associative and Cayley subbundles constructed in \cite{CaliB} allow deformations destroying the linear structure of the fiber, while the base space $L^2$ remains of the same type after twisting, namely minimal or negative superminimal.
This implies that the moduli space of calibrated submanifolds near a calibrated subbundle of this kind not only contains deformations of the base $L$ but also of the fiber. 
In contrast, the special Lagrangian bundle construction in $T^*S^n$ is much more rigid than in the case of $T^*\R^n$. 
Closing this final section, we point out potential future studies on the existence of other types of deformations in the above three cases and the possibility of finding analogous results for other manifolds of special holonomy.

We should note that \autoref{thm:assNEU} and the work of \citeauthor{Karigiannis_Lotay_2021} \cite[\Sec 5]{Karigiannis_Lotay_2021} share a similar special case. Furthermore, the latter is related to \citeauthor{Trinca_2022}'s work on $\negspinS(S^4)$ \cite[\Sec 5]{Trinca_2022}, whose findings, in turn, intersect with the results in the $\Spin(7)$ case in \cite{CaliB}. 

This paper is based on the author's master's thesis \autocite{Merkel2024} at Humboldt-Universität zu Berlin, which was supervised by Shubham Dwivedi and additionally reviewed by Thomas Walpuski. 

\begin{ack}
The author was partly supported by the Deutsche Forschungsgemeinschaft (DFG, German Research Foundation) under Germany's Excellence Strategy EXC 2044--390685587, Mathematics Münster: Dynamics--Geo\-me\-try--Structure and by the CRC 1442 \enquote{Geometry: Deformations and Rigidity} of the DFG. 

The author is grateful to Shubham Dwivedi for his encouragement and advice throughout the process of writing this paper, and thanks Hans-Joachim Hein and Thomas Walpuski for their valuable feedback. The author would also like to thank Federico Trinca for his interest and comments, as well as the anonymous referee for useful suggestions that improved the clarity and quality of the paper.
\end{ack}

\section{Review of calibrated geometry}\label{sec:calibratedgeom}

The purpose of this second section is to recall the four main examples of calibrated geometries and to collect different characterizations of their calibrated submanifolds.
See, for example, \cite{HarveyLawson1982, Joyce2007, Lotay} for more details, or \cite{CaliB} for a quick overview. 

Calibrated submanifolds are a special case of minimal submanifolds of a Riemannian manifold. 
They offer two key advantages over general minimal submanifolds:
First, they are characterized by a nonlinear partial differential equation of only first-order on the immersion map. Second, they are volume-minimizing in their homology class. 

Let $(M^n,g)$ be a Riemannian manifold of dimension $n$ and $1 \leq k \leq n-1$.
An \textbf{oriented tangent $k$-plane} on $M$ is an oriented $k$-dimensional vector subspace $V$ of some tangent space $T_{\pointM}M$ to $M$. Given such a $V$, $g\vert_V$ together with the orientation on $V$ gives a natural volume form on $V$, which we denote by $\vol_V \in \Lambda^k(V^*)$. 

\begin{defi}
	A \textbf{$k$-calibration} on $M$ is a closed $k$-form $\varphi \in \Omega^k(M)$ which satisfies $\varphi \vert_V \leq \vol_V$ for all oriented tangent $k$-planes $V$ on $M$. 
	We call a $k$-dimensional oriented immersed submanifold $N$ of $M$ \textbf{calibrated} by $\varphi$ if $\varphi\vert_{T_{\pointM}N} = \vol_{T_{\pointM}N}$ for all $\pointM \in N$.
	
	If $M$ is the total space of a vector bundle over a base $X$ and a calibrated submanifold $N$ is also the total space of a subbundle, we call $N$ a \textbf{calibrated subbundle} of $M$. 
	In this context, a \textbf{subbundle} of $M \rightarrow X$ means a vector bundle $N \rightarrow L$ over a submanifold $L$ of $X$, whose fibers are subspaces of the corresponding fibers of $M$. 
\end{defi}

Due to the \textit{holonomy principle}, manifolds of reduced holonomy are likely to possess calibrations with good prospects of calibrated submanifolds (see \cite[\Prop 2.5.2, \Ch 4.2]{Joyce2007} for details). This approach leads us to the following four main examples: 

\begin{enumerate}[label=(\Roman*)]
	\item A \textbf{Kähler manifold} is a Riemannian manifold $(M^{2n},g)$ with holonomy in $\U(n)$. Equivalently, it is a complex manifold $(M,J)$ with Hermitian metric $g$ and Kähler form $\omega$. Given such a manifold, \smash{$\omega^k/k!$} defines a calibration whose calibrated submanifolds are the complex $k$-dimensional submanifolds, i.e., those submanifolds $N \subset M$ satisfying $J(T_{\pointM}N) = T_{\pointM}N$ for every $\pointM \in N$.
	
	\item A \textbf{Calabi--Yau manifold} is a Riemannian manifold $(M^{2n},g)$ with holonomy in $\SU(n)$. Equivalently, it is a Ricci-flat Kähler manifold $(M,J,g, \omega)$ equipped with a holomorphic volume form $\Omega$. In this case, $\re(\exp(-i\theta)\Omega)$ is a calibration for every $\theta \in \R$. The corresponding calibrated submanifolds are the \textbf{special Lagrangian} submanifolds with phase $\exp(i\theta)$, which are defined as the oriented real $n$-dimensional submanifolds $N$ of $M$ that satisfy
	\begin{align*}
		\omega\vert_N = 0 \text{ (\textbf{Lagrangian})}\quad \text{ and } \quad \im\bigl(\exp(-i\theta)\Omega\bigr)\vert_N = 0 \text{ (\textbf{special} Lagrangian)}.
	\end{align*}
	
	\item A \textbf{$\Gtwo$-manifold} is a Riemannian manifold $(M^7,g)$ with holonomy in $\Gtwo$. 
	This is equivalent to the existence of a parallel $\Gtwo$-structure $\varphi \in \Omega^3(M)$ on $M$. 
	Both the \textbf{associative $3$-form} $\varphi \in \Omega^3(M)$ and the \textbf{coassociative $4$-form} $\psi = *\varphi \in \Omega^4(M)$ are calibrations and we call the corresponding calibrated submanifolds \textbf{associative $3$-folds} and \textbf{coassociative $4$-folds}, respectively.\label{item:ass}
	
	\item A \textbf{$\Spin(7)$-manifold} is a Riemannian manifold $(M^8,g)$ with holonomy in $\Spin(7)$. 
	This is equivalent to the existence of a parallel $\Spin(7)$-structure $\Phi \in \Omega^4(M)$ on $M$.
	The \textbf{Cayley $4$-form} $\Phi \in \Omega^4(M)$ defines a calibration and its calibrated submanifolds are known as \textbf{Cayley $4$-folds}. \label{item:Cayley}
\end{enumerate}

There are different characterizations of the calibrated submanifolds in examples \ref{item:ass} and \ref{item:Cayley}, and we compile some of them in the remainder of this section. 
Starting with example \ref{item:ass}, recall that there exists a (not necessarily parallel) $\Gtwo$-structure $\varphi \in \Omega^3(M)$ on $(M^7,g)$ if and only if its tangent spaces can be identified with the purely imaginary octonions $\im\,\Oct$ in a smoothly varying way.
Due to this, a $\Gtwo$-manifold $(M^7, g, \varphi)$ inherits the associator $[ \cdot, \cdot, \cdot]$ and coassociator $[ \cdot, \cdot, \cdot, \cdot]$ from $\im\,\Oct$. 
Additionally, $\varphi$ defines a natural two-fold cross product via $(u \times v)^\flat =v \lrcorner u \lrcorner \varphi$, where $\lrcorner$ denotes the interior product. 

As established in \cite[\Subsec IV.1]{HarveyLawson1982} and \cite[\Prop 2.3]{CaliB}, an oriented submanifold $E^3$ is associative in $(M^7, g, \varphi)$, that is, $\varphi\vert_E = \vol_E$, if and only if one of the following equivalent conditions (up to a change of orientation) is satisfied: 
\begin{enumerate}[label=(A\arabic*)]
	\item The tangent space $TE \subset TM$ of $E$ is preserved by the two-fold cross product. 
	\item The associator $[\cdot, \cdot, \cdot]$ vanishes on $E$. 
	\item At every point $\pointM \in E$, we have $u \lrcorner v \lrcorner w \lrcorner \psi = 0$ for some basis $\{u, v , w\}$ of $T_{\pointM}E$. \label{item:asscond}
\end{enumerate}
Furthermore, \cite[\Subsec IV.1]{HarveyLawson1982} demonstrated that an oriented submanifold $F^4$ is coassociative in $(M^7, g, \varphi)$, that is, $\psi\vert_F = \vol_F$, if and only if it fulfills one of the following equivalent conditions (up to a change of orientation): 
\begin{enumerate}[label=(B\arabic*)]
	\item The two-fold cross product $u \times v$ is orthogonal to $T_pF$ for all $u,v \in T_pF,\, p\in F$. 
	\item The coassociator $[\cdot, \cdot, \cdot, \cdot]$ vanishes on $F$. 
	\item $\varphi\vert_F = 0$.\label{item:coasscond}
\end{enumerate}

Lastly, we turn to example \ref{item:Cayley}. 
Similarly to the $\Gtwo$ case, $(M^8,g)$ admits a (not necessarily parallel) $\Spin(7)$-structure $\Phi \in \Omega^4(M)$ if and only if its tangent spaces can be smoothly identified with the octonions $\Oct$.
Thus, a $\Spin(7)$-manifold $(M^8, g, \Phi)$ acquires a four-fold cross product $\cdot \times \cdot \times \cdot \times \cdot$, and its tangent vectors split into real and imaginary parts. 
Furthermore, $\Phi$ defines a three-fold cross product $X$ via $X(u,v,w)^\flat = (- u \times v \times w)^\flat = w \lrcorner v \lrcorner u \lrcorner \Phi$, where we introduce a sign to match the convention used in \cite{CaliB}. 

By \cite[\Subsec IV.1]{HarveyLawson1982} and \cite[\Prop 2.5]{CaliB}, 
an oriented submanifold $F^4$ is Cayley in $(M^8, g, \Phi)$, that is, $\Phi\vert_F = \vol_F$, if and only if one of the following equivalent conditions (up to a change of orientation) is satisfied: 
\pagebreak
\begin{enumerate}[label=(C\arabic*)]
	\item The tangent space $TF \subset TM$ of $F$ is preserved by the three-fold cross product $X$. 
	\item At every point $\pointM \in F$, we have $\im(u \times v \times w \times y) = 0$ for some basis $\{u,v,w,y\}$ of $ T_{\pointM}F$. 
	\item At every point $\pointM \in F$, the rank $7$ bundle valued $4$-form $\eta$ on $M$ defined by
	\begin{align*}\hspace*{-15pt}
		\eta(u,v,w,y) &= u^\flat\wedge X(v,w,y)^\flat + v^\flat\wedge X(w,u,y)^\flat + w^\flat\wedge X(u,v,y)^\flat + y^\flat\wedge X(v,u,w)^\flat \\
		&\quad + u \lrcorner X(v,w,y) \lrcorner \Phi + v \lrcorner X(w,u,y) \lrcorner \Phi + w \lrcorner X(u,v,y) \lrcorner \Phi + y \lrcorner X(v,u,w) \lrcorner \Phi
	\end{align*} 
	vanishes for some basis $\{u, v , w, y\}$ of $T_{\pointM}F$. \label{item:cayleycond}
\end{enumerate}

\begin{rem}
	For a brief overview of the alternating multilinear brackets and cross products on $\Oct$, see \cite{Merkel2024}. Further details can be found in \cite[\Subsec IV.1, IV.B]{HarveyLawson1982} and \cite{OctonionsWalpuski}. 
\end{rem}

\section{The second fundamental form}\label{sec:2ndfundform}

Let $(X^n,g = \langle \cdot, \cdot \rangle)$ be a real $n$-dimensional Riemannian manifold and consider some oriented immersed submanifold $L^{\dimL}$ with immersion $x: L \rightarrow X$.
We write $(\,)^T$ and $(\,)^N$ for the orthogonal projections from $x^*(TX)$ onto the tangent bundle $dx(TL) \cong TL$ and normal bundle $NL$ of $L $ in $ X$, respectively. 
Throughout this paper, $\nabla$ always denotes the Levi-Civita connection on the tangent bundle $TX$ of the ambient manifold $X$, unless stated otherwise.

In the following sections, we examine whether certain subbundles $N \rightarrow L$ of a bundle $M \rightarrow X$ are calibrated by a given calibration on $M$. Since this is a pointwise algebraic condition on the tangent space, we have the freedom of working in coordinates specially adapted to each point in $N$. We now describe the local setup for a fixed point $p^* \in L$.

By parallel transporting orthonormal bases of $(d_{p^*}x)(T_{p^*}L)$ and $N_{p^*}L$ via the induced connections on $dx(TL)$ and $NL$, respectively, we obtain a local orthonormal frame $e_1, \dots, e_{\dimL}, \nu_{\dimL +1}, \dots, \nu_n$ for $ x^*(TX)$ that satisfies
\begin{align}\label{eq:defnormalcoords}
	(\nabla_{e_i} e_j )\vert_{p^*}^T = 0 \qquad \text{ and } \qquad (\nabla_{e_i} \nu_k )\vert_{p^*}^N = 0,
\end{align}
and which we refer to as a \textbf{normal frame} at $p^*$.
Given such a frame, we choose geodesic normal coordinates $u = (u_1, \dots, u_q)$ for $L$ centered at $p^*$ such that $\partial_{u_i}({p^*}) = (d_{p^*}x)^{-1}(e_i (p^*)) \in T_{p^*}L$ for $i=1, \dots, q$. In that case, the coordinates of the fixed point are $p^* = (0, \dots, 0)$, and the directional derivatives of the immersion $x$ at this point are given by $\partial x / \partial u_i (p^*) = (d_{p^*}x)(\partial_{u_i}(p^*)) = e_i(p^*) \in (d_{p^*}x)(T_{p^*}L)$.

From now on, we identify $dx(TL)$ with $TL$, while keeping in mind that the former is technically a subbundle of $x^*(TX)$, and is spanned by $(d_{p^*}x)(T_{p^*}L) =  \spn\{e_i(p^*) = \partial x / \partial u_i (p^*) \mid i =1, \dots, q\} \subset (x^*(TX))_{p^*} = T_{x(p^*)}X$ at the point $p^*$. Moreover, we write \smash{$TX\vert_L  \stackrel{\text{def}}{=} x^*(TX)$}
\vphantom{$\int^B$}.

Let the \textbf{second fundamental form} $A$ of the immersion $L^{\dimL} \subset X^n$ be defined as the bilinear operator
\begin{align*}
	A: \Gamma(NL) \times \Gamma(TL) \rightarrow \Gamma(TL), \quad (\nu, w) \mapsto A^\nu(w)=(\nabla_w \nu)^T.
\end{align*}
It is easy to check that for any normal vector field $\nu$, $A^\nu$ is a symmetric linear operator and, hence, diagonalizable \cite[\Sec 2]{CaliA}. We use the following abbreviations: 
\begin{align*}
	A^\nu_{ij} = \langle A^\nu(e_i), e_j \rangle = A^\nu_{ji} \qquad \text{ and } \qquad A^k_{ij} = A^{\nu_k}_{ij}.
\end{align*}

\begin{rem}
	There are different ways to define the second fundamental form. Above, we stated the definition used in \cite{HarveyLawson1982} and \cite{CaliA}, which relates to the definition \nocite{Wendl}
	\begin{align*}
		A: \Gamma(TL) \times \Gamma(TL) \rightarrow \Gamma(NL), \quad (u, w) \mapsto A(u,w)=(\nabla_w u)^N
	\end{align*}
	via $\langle A(u,w), \nu \rangle = - \langle u, A^{\nu}(w) \rangle$
	for every $\nu \in \Gamma(NL)$. 
	As the sign does not affect the results, we stick to the sign conventions and terms used in \cite{CaliA} and continue to refer to $A^\nu$ as the second fundamental form in the direction of $\nu$.
\end{rem}

\begin{rem}\label{rem:minimal}
	Recall that $L$ is minimal in $X$ if and only if $\Tr A = 0 $, which is equivalent to demanding $\Tr A^k = 0$ for all $k=\dimL +1, \dots, n$ in our notation. 
\end{rem}

Finally, using a normal frame \eqref{eq:defnormalcoords} and the second fundamental form $A$, we obtain the identities
\begin{align}\label{eq:normalcoord}
	\nabla_{e_i} e^j = - \sum_{k = \dimL +1}^{n} A^k_{ij}\nu^k \qquad \text{ and } \qquad \nabla_{e_i} \nu^k = \sum_{j=1}^{\dimL} A^k_{ij} e^j 
\end{align}
at $p^*$, where $e^1, \dots, e^{\dimL}$ and $\nu^{\dimL + 1}, \dots, \nu^n$ are the dual coframes \cite[(2.3)]{CaliA}.
	
	\section{Special Lagrangians in \texorpdfstring{$T^*S^n$}{} with the Stenzel metric} \label{sec:sLagNEW}
	
We begin by examining special Lagrangians in the cotangent space $T^*S^n$ of the standard round sphere. 
To do so, we first endow $T^*S^n$ with a Calabi--Yau structure following \cite{Szoke1991, Stenzel1993, Anciaux2003, CaliB}: We identify $T^*S^n \subset \R^{n+1} \oplus (\R^{n+1})^*$ with the complex quadric 
\begin{align*}
	Q = \Bigl\{(z_0, \dots, z_n) \in \C^{n+1}\ \Big| \ \sum_{k=0}^n z_k^2 = 1 \Bigr\}
\end{align*}
via the diffeomorphism 
\begin{align}\label{sLagdiffeo}
	\Psi: T^*S^n \rightarrow Q, \quad 
	(x, \xi) \mapsto x\cosh\lvert \xi \rvert + i \frac{\xi}{\lvert \xi \rvert } \sinh\lvert \xi \rvert,
\end{align}
where the term $\frac{\xi}{\lvert \xi \rvert } \sinh\lvert \xi \rvert$ is to be interpreted as $0$ when $|\xi| = 0$, and we identify $\C^{n+1} = \R^{n+1} \oplus i (\R^{n+1})^*$.
From this, $T^*S^n$ inherits a natural complex structure, a nowhere vanishing holomorphic $(n,0)$-form $\Omega$ and the Stenzel metric, turning it into a Calabi--Yau manifold.
By \cite[Lemma 2.1]{Anciaux2003}, the corresponding Kähler form is given by 
\begin{gather}
	\omega_{\text{St}} = \frac{i}{2} \sum_{j,k = 1}^n a_{jk} dz_j \wedge d\bar{z}_k \nonumber
	\intertext{with}
	a_{jk} = \left(\delta_{jk} + \frac{z_j \bar{z}_k}{\lvert z_0 \rvert^2} \right)v^\prime + 2\, \re \left(\overline{z}_jz_k - \frac{\bar{z}_0}{z_0}z_jz_k \right)v^{\prime\prime}\label{eq:Stenzelaij}
\end{gather}
in a neighborhood of a point where $z_0 \neq 0$. 
Above, $v$ is a function of $r =\lvert z \rvert$, defined by a certain differential equation which ensures that the metric is Ricci-flat. For our purposes, it suffices to know that $v^\prime(r), v^{\prime\prime}(r) > 0$ for $r > 0$ \cite[\Prop 6]{Stenzel1993}. 
By construction, the Kähler potential only depends on $r$, which implies that $\omega_{\text{St}}$ is $\Orm_{n+1}(\C)$-invariant.

Let $L^{\dimL} \subset S^n$ be an oriented immersed submanifold and $\mu \in \Omega^1(L)$ be a $1$-form on $L$. We define the space 
\begin{align*}
	X_\mu = \bigl\{(\pointL, \xi + \mu(\pointL)) \in T^*S^n \vert_L \, \big| \, \pointL \in L,\, \xi \in N^*_{\pointL}L\bigr\},
\end{align*}
which we sometimes mnemonically refer to it as \enquote{$N^*L + \mu$}. 
This is a \enquote{twisting} of the conormal bundle $N^*L$ over $L$ obtained by affinely translating each fiber $N^*_{\pointL}L$ by a cotangent vector $\mu({\pointL}) \in T^*_{\pointL}L$. 
Our goal is to find conditions on $L$ and $\mu$ so that $N^*L + \mu$ is special Lagrangian in $T^*S^n$. We start with the conditions for it to be Lagrangian.

\begin{satz}\label{prop:Lag}
	The submanifold $N^*L + \mu$ is Lagrangian in $T^*S^n$ if and only if $\mu = 0$. 
\end{satz}

\begin{proof}[Proof.]
	The submanifold $N^*L + \mu$ is Lagrangian in $T^*S^n$ if and only if its tangent space at every point is a Lagrangian subspace of the corresponding tangent space to $T^*S^n$. We will show that for this to be the case at a point $(p, \xi + \mu(p)) \in (N^*L \setminus L) + \mu$, the 1-form $\mu $ has to vanish at $p$. The other implication then follows from the proof of \cite[\Thm 3.1]{CaliB}. 
	
	To do so, we fix a point $(\pointL^*, \xi^* + \mu({\pointL^*})) \in X_\mu$ with $\xi^* \neq 0$, and let $(e^1, \dots, e^\dimL, \nu^{\dimL+1}, \dots, \nu^n)$ be a local orthonormal adapted coframe along $L$ whose dual frame is normal at $p^*$ (see \eqref{eq:defnormalcoords}). Moreover, we let $u = (u_1, \dots, u_{\dimL})$ be the coordinates for $L$ described in \autoref{sec:2ndfundform}, and $t = (t_{\dimL +1}, \dots, t_n)$ be the coordinates on the fiber with respect to the chosen coframe.
	Then our fixed point $(\pointL^*, \xi^* + \mu({\pointL^*})) $ has the coordinates $(0, t^*)$ for some $t^*= (t_{\dimL +1}^*, \dots, t_n^*) \neq 0$. 
	
	The immersion of $N^*L + \mu$ into $T^*S^n$, $\Phi: X_\mu \rightarrow T^*S^n$, is locally given by
	\begin{align*}
		\Phi: (u,t)\mapsto \Bigl(x(u), \sum_{k=\dimL +1}^n t_k \nu^k(u) + \mu(u)\Bigr) = \Bigl(x(u), \sum_{k=\dimL +1}^n t_k \nu^k(u) + \sum_{l=1}^{\dimL} a_l(u)e^l(u)\Bigr),
	\end{align*}
	where $x$ is the immersion of $L$ into $S^n$, and $a = (a_1, \dots, a_{\dimL})$ are the coordinates of $\mu$ with respect to the local trivialization \smash{$T^*L \stackrel{\text{loc}}{=}  \spn\{ e^1, \dots, e^{\dimL} \} $}\vphantom{$\int^B$}. 
	
	Define $\hat{\nu}(u,t) = \sum_{k=\dimL +1}^n t_k \nu^k(u)$ and $y(u,t) = \lvert \hat{\nu}(u,t) + \mu(u) \rvert^2$. 
	Since $e^1, \dots, e^{\dimL},\allowbreak \nu^{\dimL +1}, \dots, \nu^n$ are orthonormal, we get 
	\begin{align}\label{eq:y}
		y(u,t) 
		&= \Big\lvert\sum_{k=\dimL +1}^n t_k \nu^k(u) + \sum_{l=1}^{\dimL}a_l(u)e^l(u) \Big\rvert^2 = \sum_{k=\dimL +1}^n\lvert t_k\rvert^2 + \sum_{l=1}^{\dimL}\lvert a_l(u)\rvert^2 \nonumber \\
		&= \lvert t\rvert^2 + \lvert a(u)\rvert^2. 
	\end{align}
	In particular, $y$ vanishes nowhere in a sufficiently small neighborhood of $(0, t^*)$ because $t^* \neq 0$.
	Restricting the diffeomorphism $\Psi: T^*S^n \rightarrow Q$ (see \eqref{sLagdiffeo}) to $\Phi(X_\mu) \subset T^*S^n$ gives 
	\begin{align}\label{eq:PsionPhi}
		\Psi(x, \hat{\nu} + \mu) 
		= x\cosh\sqrt{y} + i\, \frac{\hat{\nu} + \mu}{\sqrt{y}} \sinh\sqrt{y}.
	\end{align}
	
	We want to determine when the Kähler form $\omega_{\text{St}}$ vanishes on the tangent space to $(\Psi \circ \Phi)(X_\mu)$ at $z^* =  (\Psi \circ \Phi)(0,t^*) \in Q \subset \C^{n+1}$.
	Let $(\e_0, \dots, \e_n)$ denote the orthonormal basis of $\R^{n+1}$ with dual basis $(\e^0, \dots, \e^n)$ defined by $(x, e^1, \dots, e^{\dimL}, \nu^{\dimL +1}, \dots, \nu^n)\vert_{u=0} = (\e_0, \e^1, \dots, \e^n)$. 
	Then the point $z^*$ takes the form
	\begin{align*}
		z^*
		&=\cosh\sqrt{y(0, t^*)}\, \e_0 + i\, \frac{\sum_{k=q+1}^{n} t_k^*\e^k + \sum_{l=1}^{\dimL} a_l(0)\e^l}{\sqrt{y(0, t^*)}} \sinh\sqrt{y(0, t^*)},
	\end{align*}
	which simplifies to
	\begin{align*}
		z^* = \cosh\sqrt{y}\, \e_0 + i\sum_{l=1}^{\dimL} \biggl( \frac{\sinh\sqrt{y} }{\sqrt{y}} a_l \biggr)\e^l + i\sum_{k=q+1}^{n} \biggl( \frac{\sinh\sqrt{y} }{\sqrt{y}} t_k^* \biggr)\e^k ,
	\end{align*}
	omitting the dependence on $(0,t^*)$. 
	
	The invariance of $\omega_{\text{St}}$ under $\Orm_{n+1}(\C)$ allows us to transform the standard coordinates on $\C^{n+1}$ via (complex) orthogonal transformations without changing the form's local expression (as long as the coefficients $a_{jk}$ remain well-defined). Explicitly, we can use $(\e_0, \dots, \e_n)$ as the underlying orthonormal basis for $\R^{n+1}$ with dual basis $(\e^0, \dots, \e^n)$. The corresponding coordinates for $\C^{n+1} = \R^{n+1} \oplus i (\R^{n+1})^*$ are then given by $z_k = \langle \re\, \cdot, \e_k \rangle_{\R^{n+1}} + i \langle \im\, \cdot, \e^k \rangle_{(\R^{n+1})^*}$ for $k = 0, \dots, n$, and are related to the standard coordinates on $\C^{n+1}$ via a (real) orthogonal transformation. As these functions are linear, they satisfy $dz_k = z_k$ for all $k= 0, \dots, n$. 
	
	As a result, the point $z^*$ has the coordinates
	\begin{align*}
		z_0 =  \cosh\sqrt{y},\quad z_l =  i\, \frac{\sinh\sqrt{y} }{\sqrt{y}} a_l\quad(l=1, \dots, q), \quad z_k = i\, \frac{\sinh\sqrt{y}}{\sqrt{y}} t^*_k \quad(k=q+1, \dots, n).
	\end{align*}
	Therefore, the coefficients of the Kähler form $\omega_{\text{St}}$ corresponding to the Stenzel metric (see \eqref{eq:Stenzelaij}) are given by 
	\begin{align*}
		a_{kl} = a_{lk}
		= 
		\begin{dcases}
			\biggl(\delta_{kl}+ \frac{a_ka_l}{y} \tanh^2\sqrt{y}\biggr)v^\prime + \frac{4a_ka_l}{y}\sinh^2\sqrt{y} \, v^{\prime\prime},& k \leq l \leq \dimL, \\
			\frac{a_kt_l^*}{y} \tanh^2\sqrt{y}\, v^\prime + \frac{4a_kt^*_l}{y}\sinh^2\sqrt{y} \, v^{\prime\prime},& k\leq \dimL < l, \\
			\biggl(\delta_{kl}+ \frac{t^*_kt^*_l}{y} \tanh^2\sqrt{y}\biggr)v^\prime + \frac{4t^*_kt^*_l}{y}\sinh^2\sqrt{y} \, v^{\prime\prime},& \dimL < k \leq l
		\end{dcases}
	\end{align*}
	at the point $z^*$. 
	Their symmetry allows us to write $\omega_{\text{St}}$ as
	\begin{align}\label{wst}
		\omega_{\text{St}} 
		&=   \frac{i}{2} \sum_{k=1}^n a_{kk} (dz_k \wedge d\bar{z}_k) + \frac{i}{2} \sum_{1 \leq k < l \leq n} a_{kl} (dz_k \wedge d\bar{z}_l + dz_l \wedge d\bar{z}_k).
	\end{align}
	
	The tangent space $T_{z^*}((\Psi \circ \Phi)(X_\mu))$ is spanned by $E_i = ( \Psi \circ \Phi)_ *(\partial_{ u_i})$ and $F_j = ( \Psi \circ \Phi)_ *(\partial_{t_j})$ for $i = 1, \dots, \dimL$ and $ j = \dimL +1, \dots, n$. Specifically, $E_i$ is given by
	\begin{align*}
		E_i&= \Bigg[\frac{\sinh \sqrt{y} }{\sqrt{y}} \biggl( \sum_{l=1}^{\dimL} a_l\frac{\partial a_l}{\partial {u_i}} \biggr)  x + \cosh\sqrt{y}\, \frac{\partial x}{\partial {u_i}}  \Bigg. \\*
		&\qquad + i\, \frac{1}{y} \biggl( \sum_{l=1}^{\dimL} a_l \frac{\partial a_l}{\partial {u_i}} \biggr) \biggl(\cosh\sqrt{y} - \frac{\sinh\sqrt{y}}{\sqrt{y}}\biggr) (\hat{\nu} + \mu)
		\Bigg.+ i \, \frac{\sinh\sqrt{y}}{\sqrt{y}} \left(\nabla_{e_i} \hat{\nu} + \nabla_{e_i}\mu\right) \Bigg]\Bigg\vert_{(0,t^*)}
	\end{align*}
	for all $i=1, \dots, q$. 
	As we are working with a normal frame, we can use \eqref{eq:normalcoord} to compute
	\begin{align*}
		\nabla_{e_i} \hat{\nu} &= \nabla_{e_i} \Bigl(\sum_{k=\dimL +1}^n t_k \nu^k \Bigr) = \sum_{k=\dimL +1}^n t_k (\nabla_{e_i}\nu^k)
		= \sum_{k=\dimL +1}^n t_k \Bigl(\sum_{l=1}^{\dimL} A^k_{il} e^l \Bigr) \\*
		&= \sum_{l=1}^{\dimL} \Bigl(\sum_{k=\dimL +1}^n t_k A^k_{il}\Bigr) e^l 
		= \sum_{l=1}^{\dimL} A^{\hat{\nu}}_{il} e^l 
		= \sum_{l=1}^{\dimL} A^{\hat{\nu}}_{il} \e^l 
	\end{align*}
	and 
	\begin{align*}
		\nabla_{e_i} \mu 
		&= \nabla_{e_i} \Bigl(\sum_{l=1}^{\dimL} a_le^l\Bigr) 
		= \sum_{l=1}^{\dimL} \biggl(\frac{\partial a_l}{\partial u_i} e^l + a_l \nabla_{e_i} e^l\biggr) 
		= \sum_{l=1}^{\dimL} \biggl(\frac{\partial a_l}{\partial u_i} e^l - \sum_{k = \dimL +1}^{n} a_lA^k_{il}\nu^k\biggr) \\*
		&= \sum_{l=1}^{\dimL} \frac{\partial a_l}{\partial u_i} e^l - \sum_{k = \dimL +1}^{n} \Bigl( \sum_{l=1}^{\dimL} a_lA^k_{il}\Bigr)\nu^k 
		= \sum_{l=1}^{\dimL} \frac{\partial a_l}{\partial u_i} \e^l - \sum_{k = \dimL +1}^{n} \Bigl( \sum_{l=1}^{\dimL} a_lA^k_{il}\Bigr)\e^k 
	\end{align*}
	at $(0, t^*)$.
	Thus, every $E_i$ takes the form
	\begin{align*}
		E_i 
		&= \frac{\sinh \sqrt{y} }{\sqrt{y}} \biggl( \sum_{l=1}^{\dimL} a_l\frac{\partial a_l}{\partial {u_i}} \biggr)  \e_0 + \cosh\sqrt{y}\, \e_i \\*
		&\quad+ i\, \frac{1}{y} \biggl( \sum_{l=1}^{\dimL} a_l \frac{\partial a_l}{\partial {u_i}} \biggr) \biggl(\cosh\sqrt{y} - \frac{\sinh\sqrt{y}}{\sqrt{y}}\biggr) \Bigl( \sum_{l=1}^{\dimL} a_l \e^l+ \sum_{k=q+1}^{n} t_k^*\e^k\Bigr)\\*
		&\quad+ i\, \frac{\sinh\sqrt{y}}{\sqrt{y}} \Biggl(\sum_{l=1}^{\dimL} \biggl(A^{\hat{\nu}}_{il} +\frac{\partial a_l}{\partial u_i} \biggr)\e^l - \sum_{k = \dimL +1}^{n} \Bigl( \sum_{l=1}^{\dimL} a_lA^k_{il}\Bigr)\e^k \Biggr).
	\end{align*}
	
	On the other hand, $F_j = (\Psi \circ \Phi)_ *(\partial_{t_j})$ is given by
	\begin{align*}
		F_j&= \Bigg[t_j \frac{\sinh \sqrt{y} }{\sqrt{y}}x
		+ i\, \frac{t_j}{y} \biggl(\cosh\sqrt{y} - \frac{\sinh\sqrt{y}}{\sqrt{y}}\biggr) (\hat{\nu} + \mu)
		+ i \, \frac{\sinh\sqrt{y}}{\sqrt{y}} \nu^j\Bigg]\Bigg\vert_{(0,t^*)} \\
		&= t_j^* \frac{\sinh \sqrt{y} }{\sqrt{y}}\e_0
		+ i\, \frac{t_j^*}{y} \biggl(\cosh\sqrt{y} - \frac{\sinh\sqrt{y}}{\sqrt{y}}\biggr) \Bigl(\sum_{l=1}^{\dimL} a_l \e^l+ \sum_{k=q+1}^{n} t_k^*\e^k\Bigr)
		+ i \, \frac{\sinh\sqrt{y}}{\sqrt{y}} \e^j 
	\end{align*}
	for all $j=\dimL +1, \dots, n$. 
	
	Our goal is to compute $\omega_\text{St}$ of any pair of vectors in the basis $E_1, \dots, E_q,\allowbreak F_{q+1}, \dots, F_n$. 
	To do so, we first need to determine $dz_k$ and $d\bar{z}_k$ of every basis vector for all $k=1, \dots, n$. 
	(Recall that $dz_k= z_k = \langle \re\, \cdot, \e_k \rangle_{\R^{n+1}} + i \langle \im\, \cdot, \e^k \rangle_{(\R^{n+1})^*}$ for all $k$.)
	For every $i=1, \dots, q$, we find
	\begin{align*}
		dz_k(E_i) &= \overline{d\bar{z}_k(E_i)} \\*
		&=\ \delta_{ik} \cosh\sqrt{y} \\*
		&\quad+ i\, \Biggl(\frac{a_k}{y} \biggl( \sum_{l=1}^{\dimL} a_l \frac{\partial a_l}{\partial {u_i}} \biggr) \biggl(\cosh\sqrt{y} - \frac{\sinh\sqrt{y}}{\sqrt{y}}\biggr)
		+ \biggl(A^{\hat{\nu}}_{ik} + \frac{\partial a_k}{\partial u_i}\biggr)\frac{\sinh\sqrt{y}}{\sqrt{y}}
		\Biggr)
		\intertext{for $k \leq \dimL$, and}
		%%%
		%%%
		dz_{k}(E_i) &= -d\bar{z}_{k}(E_i) \\*
		&= i\,\Biggl(\frac{ t^*_k}{y} \biggl( \sum_{l=1}^{\dimL} a_l \frac{\partial a_l}{\partial {u_i}} \biggr) \biggl(\cosh\sqrt{y} - \frac{\sinh\sqrt{y}}{\sqrt{y}}\biggr)
		- \Bigl(\sum_{l=1}^{\dimL}a_lA^{k}_{il}\Bigr)\frac{\sinh\sqrt{y}}{\sqrt{y}}\Biggr) 
		\intertext{for $k \geq \dimL +1$. The remaining basis vectors satisfy}
		%%%
		%%%
		dz_k(F_{j}) &= - d\bar{z}_k(F_{j})\\*
		&= i\, \frac{a_kt_j^*}{y} \biggl(\cosh\sqrt{y} - \frac{\sinh\sqrt{y}}{\sqrt{y}}\biggr) 
		\intertext{for $k \leq \dimL$, and}
		dz_k(F_{j}) &= - d\bar{z}_k(F_{j})\\*
		&= i\, \frac{t^*_kt_j^*}{y} \biggl(\cosh\sqrt{y} - \frac{\sinh\sqrt{y}}{\sqrt{y}}\biggr) +  i \delta_{jk} \frac{\sinh\sqrt{y}}{\sqrt{y}}
	\end{align*}
	for $k \geq \dimL +1$, where $j = q+1, \dots, n$.
	
	We start with the pair $(F_j, E_i)$ for $i=1, \dots, q$ and $j = q+1, \dots, n$. Using the above formulas, we compute 
	\begin{align*}
		&(dz_k \wedge d\bar{z}_l + dz_l \wedge d\bar{z}_k)(F_{j}, E_i) \\*
		&\qquad= dz_k(F_j)\bigl(d\bar{z}_l(E_i) + dz_l(E_i)\bigr) + dz_l(F_j)\bigl(dz_k(E_i) + d\bar{z}_k(E_i)\bigr) \\
		&\qquad= 2 \Bigl(dz_k(F_j)\, \re\bigl(dz_l(E_i)\bigr) + dz_l(F_j)\, \re\bigl(dz_k(E_i)\bigr)\Bigr) \\
		&\qquad= 2 i\, \frac{t_j^* \cosh\sqrt{y}}{y} \biggl(\cosh\sqrt{y} - \frac{\sinh\sqrt{y}}{\sqrt{y}}\biggr) (a_k\delta_{li} + a_l \delta_{ki})
		\intertext{for $k \leq l \leq \dimL$,}
		%%%
		%%%
		&(dz_k \wedge d\bar{z}_l + dz_l \wedge d\bar{z}_k)(F_{j}, E_i) \\*
		&\qquad= 2 \Bigl(dz_k(F_j)\, \re\bigl(dz_l(E_i)\bigr) + dz_l(F_j)\, \re\bigl(dz_k(E_i)\bigr)\Bigr) \\
		&\qquad= 2i\, \delta_{ki} \cosh\sqrt{y}\,\Biggl(\frac{t^*_lt_j^*}{y} \biggl(\cosh\sqrt{y} - \frac{\sinh\sqrt{y}}{\sqrt{y}}\biggr) + \delta_{jl} \frac{\sinh\sqrt{y}}{\sqrt{y}}\Biggr)\\
		&\qquad= 2i\, \delta_{ki}\frac{t^*_lt_j^* \cosh\sqrt{y}}{y} \biggl(\cosh\sqrt{y} - \frac{\sinh\sqrt{y}}{\sqrt{y}}\biggr) 
		+  2i\, \delta_{ki}\delta_{jl} \frac{\cosh\sqrt{y}\sinh\sqrt{y}}{\sqrt{y}}
		\intertext{for $k \leq q < l$, and}
		%%%
		%%%
		&(dz_k \wedge d\bar{z}_l + dz_l \wedge d\bar{z}_k)(F_{j}, E_i) \\*
		&\qquad= 2 \Bigl(dz_k(F_j)\, \re\bigl(dz_l(E_i)\bigr) + dz_l(F_j)\, \re\bigl(dz_k(E_i)\bigr)\Bigr) = 0
	\end{align*}
	for $q <  k \leq l$. 
	Substituting these formulas into \eqref{wst} yields
	\begin{align*}
		\omega_\text{St}(E_i, F_j) 
		%%%
		%%%
		&=  -\frac{i}{2} \sum_{k=1}^q a_{kk} (dz_k \wedge d\bar{z}_k )(F_j, E_i) \\
		&\quad
		- \frac{i}{2} \sum_{1 \leq k < l \leq q} a_{kl} (dz_k \wedge d\bar{z}_l + dz_l \wedge d\bar{z}_k)(F_j, E_i) \\
		&\quad- \frac{i}{2} \sum_{1 \leq k\leq q < l \leq n} a_{kl} (dz_k \wedge d\bar{z}_l + dz_l \wedge d\bar{z}_k)(F_j, E_i) \\
		%%%
		%%%
		&=  \frac{t_j^* \cosh\sqrt{y}}{y} \biggl(\cosh\sqrt{y} - \frac{\sinh\sqrt{y}}{\sqrt{y}}\biggr) \\*
		&\quad \cdot \Bigl(
		\sum_{k=1}^q a_{kk}a_k\delta_{ki} 
		+ \sum_{1 \leq k < l \leq q} a_{kl} (a_k\delta_{li} + a_l \delta_{ki})  		
		+ \sum_{1 \leq k\leq q < l \leq n} a_{kl} \delta_{ki} t^*_l  
		\Bigr)\\*
		&\quad + \frac{\cosh\sqrt{y}\sinh\sqrt{y}}{\sqrt{y}} \sum_{1 \leq k\leq q < l \leq n} a_{kl} \delta_{ki}\delta_{jl}  \\
		%%%
		%%%
		&=  \frac{t_j^* \cosh\sqrt{y}}{y} \biggl(\cosh\sqrt{y} - \frac{\sinh\sqrt{y}}{\sqrt{y}}\biggr) 
		\Bigl(\sum_{k=1}^q a_{ki}a_k + \sum_{l = q+1}^n a_{il} t^*_l \Bigr) \\*
		&\quad + \frac{\cosh\sqrt{y}\sinh\sqrt{y}}{\sqrt{y}} a_{ij}  \\
		%%%
		%%%
		&=  \frac{t_j^* \cosh\sqrt{y}}{y} \biggl(\cosh\sqrt{y} - \frac{\sinh\sqrt{y}}{\sqrt{y}}\biggr) \\
		&\quad \cdot \Biggl[\sum_{k=1}^q \biggl(\biggl(\delta_{ki}a_k+ \frac{a_k^2a_i}{y} \tanh^2\sqrt{y}\biggr)v^\prime + \frac{4a_k^2a_i}{y}\sinh^2\sqrt{y} \, v^{\prime\prime}\biggr) \\
		&\qquad + \sum_{l = q+1}^n \biggl( \frac{a_i(t_l^*)^2}{y} \tanh^2\sqrt{y}\,v^\prime + \frac{4a_i(t^*_l)^2}{y}\sinh^2\sqrt{y} \, v^{\prime\prime}\biggr) \Biggr]\\
		&\quad + \frac{\cosh\sqrt{y}\sinh\sqrt{y}}{\sqrt{y}} \biggl(\frac{a_it_j^*}{y} \tanh^2\sqrt{y}\, v^\prime + \frac{4a_it^*_j}{y}\sinh^2\sqrt{y} \, v^{\prime\prime}\biggr)  \\
		%%%
		%%%
		&=  \frac{a_i t_j^* \cosh\sqrt{y}}{y} \biggl(\cosh\sqrt{y} - \frac{\sinh\sqrt{y}}{\sqrt{y}}\biggr) \\*
		&\quad \cdot \Biggl[\biggl(1+ \frac{\lvert a\rvert^2 }{y} \tanh^2\sqrt{y}\biggr)v^\prime + \frac{4\lvert a \rvert^2 }{y}\sinh^2\sqrt{y} \, v^{\prime\prime} \\*
		&\qquad + \biggl(\frac{\lvert t^*\rvert^2}{y} \tanh^2\sqrt{y}\, v^\prime + \frac{4\lvert t^*\rvert^2}{y}\sinh^2\sqrt{y} \, v^{\prime\prime}\biggr) \Biggr]\\*
		&\quad + \frac{a_it_j^*\cosh\sqrt{y}}{y} \frac{\sinh\sqrt{y}}{\sqrt{y}} \bigl( \tanh^2\sqrt{y}\, v^\prime + 4\sinh^2\sqrt{y} \, v^{\prime\prime}\bigr)  \\
		%%%
		%%%
		&=  \frac{a_it_j^* \cosh\sqrt{y}}{y} \biggl(\cosh\sqrt{y} - \frac{\sinh\sqrt{y}}{\sqrt{y}}\biggr) 
		\bigl( v^\prime + \tanh^2\sqrt{y}\, v^\prime + 4 \sinh^2\sqrt{y} \, v^{\prime\prime} \bigr)\\
		&\quad + \frac{a_it_j^* \cosh\sqrt{y}}{y} \frac{\sinh\sqrt{y}}{\sqrt{y}} \bigl( \tanh^2\sqrt{y}\, v^\prime + 4\sinh^2\sqrt{y} \, v^{\prime\prime}\bigr)  \\
		%%%
		%%%
		&=  \frac{a_it_j^* \cosh\sqrt{y}}{y} \biggl(\cosh\sqrt{y}\,
		\bigl( v^\prime + \tanh^2\sqrt{y}\, v^\prime + 4 \sinh^2\sqrt{y} \, v^{\prime\prime} \bigr) - \frac{\sinh\sqrt{y}}{\sqrt{y}} v^\prime \biggr)\\
		%%%
		%%%
		&=  \frac{a_it_j^* \cosh^2\sqrt{y}}{y} \biggl(
		\biggl( 1 - \frac{\tanh\sqrt{y}}{\sqrt{y}} + \tanh^2\sqrt{y}\biggr) v^\prime + 4 \sinh^2\sqrt{y} \, v^{\prime\prime} \biggr)
	\end{align*}
	for $i= 1, \dots, \dimL$ and $j=q+1, \dots, n$. 
	Since $t^*$ is nonzero by assumption, $y$ is positive, which further implies that $v^\prime, v^{\prime\prime} > 0$ \cite[\Prop 6]{Stenzel1993} and {$1 - \frac{\tanh \sqrt{y}}{\sqrt{y}} + \tanh^2\sqrt{y} > 0$}. Additionally, this ensures that $t_j^* \neq 0$ for at least one $j \in \{q+1, \dots, n\}$. Consequently, $\omega_\text{St}(E_i, F_j)$ vanishes for all $j = q+1, \dots, n$ if and only if $a_i=0$. 
	In other words, we have $\omega_\text{St}(E_i, F_j) = 0$ for all $i=1, \dots, \dimL$ and $j=q+1, \dots, n$ if and only if $\mu(p^*) =  \sum_{l=1}^{\dimL} a_l(0) \e^l = 0$.
	
	As $p^* \in L$ was arbitrary, $\mu = 0 \in \Omega^1(L)$ is a necessary condition for $N^*L + \mu$ to be Lagrangian. According to the proof of \cite[\Thm 3.1]{CaliB}, it is also sufficient.
\end{proof}

\autoref{prop:Lag} implies that twisting $N^*L$ by $1$-forms on $L$ does not provide any new examples. In fact, all possible special Lagrangians constructed in this way were already described in \cite{CaliB}. This shows that special Lagrangians of the form $N^*L$ in $T^*S^n$ are much more rigid than those in $T^*\R^n$. Nevertheless, let us capture this result in the following corollary. 

\begin{cor}
	The submanifold $N^*L + \mu$ is special Lagrangian in $T^*S^n$ if and only if $L^{\dimL}$ is austere in $S^n$ and $\mu = 0$. 
\end{cor}

\section{(Co-)associative submanifolds of \texorpdfstring{$\Lambda^2_{-}(T^*X)$ with the Bryant--Salamon metric for $X^4 = S^4, \mathbb{CP}^2$}{}}\label{sec:coassNEW}

Let us proceed to associative and coassociative submanifolds of the space of anti-self-dual $2$-forms $\Lambda^2_{-}(T^*X)$, where $X^4$ represents either the $4$-di\-men\-sio\-nal sphere $S^4$ or the complex projective plane $\mathbb{CP}^2$ with the standard metrics.
To begin with, we equip $M^7 = \Lambda_{-}^2(T^*X^4)$ with a parallel $\Gtwo$-structure following \cite{BryantSalamon1989}. 

We equip $M^7$ with the connection $\nabla$ induced by the Levi-Civita connection on $(X^4,g)$. This provides a canonical splitting of its tangent space $T_{\omega} M \cong \mathcal{H}_{\omega} \oplus \mathcal{V}_{\omega}$ into the horizontal and vertical spaces for any $\omega \in M$. We can identify $\mathcal{H}_{\omega}$ with the tangent space $T_{\pi(\omega)}X$ of $X$ via the linear isomorphism
$\Hor_{\omega} = (\pi_*\vert_{\mathcal{H}_{\omega}})^{-1}: T_{\pi(\omega)}X \rightarrow \mathcal{H}_{\omega}$, where $\pi: M^7 = \Lambda_{-}^2(T^*X) \rightarrow X$ denotes the projection onto the base.
On the other hand, $\mathcal{V}_{\omega}$ can be identified with the fiber \smash{$M_{\pi(\omega)} = \Lambda_{-}^2(T^*_{\pi(\omega)}X)$} through the linear isomorphism {$\Ver_{\omega}: M_{\pi(\omega)} \rightarrow \mathcal{V}_{\omega} = T_\omega (M_{\pi(\omega)} ),\, \sigma \mapsto \frac{d}{dt} (\omega +t \sigma)\vert_{t=0}$}.
Due to these identifications, the metric $g$ on $X$ induces metrics \smash{$g_{\mathcal{H}}$} and \smash{$g_{\mathcal{V}}$} with canonical volume forms \smash{$\vol_{\mathcal{H}}$} and \smash{$\vol_{\mathcal{V}}$} on $\mathcal{H}$ and $\mathcal{V}$, respectively. 
\nocite{Wendl}

\begin{satz}[{\cite[\Thm 4.1]{BryantSalamon1989}}]\label{BSEF}
	Let $(X^4,g)$ be either $S^4$ with the standard round metric or $\mathbb{CP}^2$ with the Fubini-Study metric, and let $r$ denote the radial coordinate in the vertical fibers. Then there exist positive functions $u= u(r)$ and $v=v(r)$ such that 
	\begin{align*}
		g_{M^7} = u^2 g_{\mathcal{H}} \oplus v^2 g_{\mathcal{V}}
	\end{align*}
	defines a complete metric on $M^7 = \Lambda_{-}^2(T^*X^4)$ with holonomy equal to $\Gtwo$. Its fundamental $3$-form $\varphi$ is given by
	\begin{align*}
		\varphi = v^3\, \vol_{\mathcal{V}} + u^2 v\, d\theta,
	\end{align*}
	where $\theta_\omega = \pi^*\omega$ ($\omega \in M$) is the canonical soldering $2$-form on $M$. 
\end{satz}

Let us now restrict the vector bundle $M^7 = \Lambda_{-}^2(T^*X^4) \rightarrow X$ to an oriented immersed submanifold $L^2 \subset X^4$ and fix some oriented local orthonormal adapted frame $(e_1, e_2, \nu_3, \nu_4)$ along $L$ with dual coframe \smash{$(e^1, e^2, \nu^3, \nu^4)$}.
The anti-self-dual $2$-forms $f^1 = e^1 \wedge e^2 - \nu^3 \wedge \nu^4$, $f^2 = e^1 \wedge \nu^3 - \nu^4 \wedge e^2$ and $f^3 = e^1 \wedge \nu^4 - e^2 \wedge \nu^3$ locally trivialize $\Lambda_{-}^2(T^*X)\vert_L$. We denote the horizontal lifts of the tangent and normal vectors in $TX\vert_L$ to $\mathcal{H}$ by $\bar{e}_i = \Hor\,e_i$, $\bar{\nu}_j = \Hor\,\nu_j,\, i=1,2, j=3, 4$, and the vertical lifts of the anti-self-dual 2-forms on $X$ to $\mathcal{V}$ by $\check{f}^k,\, k=1, 2, 3$. Furthermore, we refer to their dual horizontal and vertical $1$-forms as $\bar{e}^i, \bar{\nu}^j$ and $\check{f}_k$, respectively.
The following diagram provides a compact illustration of the situation described for an $\omega \in \pi^{-1}(L) \subset M$: 
\noindent\begin{equation*}
	\begin{tikzcd}[column sep=-0.5em, row sep=1em]
		& \spn \{\bar{e}_1, \bar{e}_2, \bar{\nu}_3, \bar{\nu}_4\}\vert_{\omega} \arrow[d, equal, shorten=1mm] && \spn \{\check{f}^1, \check{f}^2, \check{f}^3\}\vert_{\omega} \arrow[d,equal, shorten=1mm]\\
		T_{\omega} M = T_{\omega}\bigl(\Lambda_{-}^2 (T^*X)\bigr) \cong &\mathcal{H}_{\omega} &\oplus& \mathcal{V}_{\omega} \\
		&&&\\
		&T_{\pi(\omega)}X \arrow[uu, "\Hor_{\omega}", "\cong"'] && M_{\pi(\omega)} \arrow[uu, "\Ver_{\omega}", "\cong"'] = \Lambda_{-}^2(T^*_{\pi(\omega)}X) \\
		&\spn \{e_1, e_2, \nu_3, \nu_4\}\vert_{\pi(\omega)} \arrow[u, equal, shorten=1mm] && \spn \{f^1, f^2, f^3\}\vert_{\pi(\omega)} \arrow[u,equal, shorten=0.5mm]
	\end{tikzcd}
\end{equation*}
As a result, the fundamental 3-form $\varphi$ in \autoref{BSEF} and its Hodge dual $\psi = * \varphi$ restricted to $L$ are locally given by 
\begin{align}
	\varphi &= v^3(\check{f}_1\wedge \check{f}_2 \wedge \check{f}_3) + u^2v\, \check{f}_1 \wedge (\bar{e}^1\wedge \bar{e}^2 - \bar{\nu}^3\wedge \bar{\nu}^4) \nonumber\\*
	&\quad + u^2v\, \check{f}_2 \wedge (\bar{e}^1\wedge \bar{\nu}^3 - \bar{\nu}^4\wedge \bar{e}^2) + u^2v\, \check{f}_3 \wedge (\bar{e}^1\wedge \bar{\nu}^4 - \bar{e}^2\wedge \bar{\nu}^3) \label{eq:varphi}
	\intertext{and} 
	\psi &= u^4(\bar{e}^1\wedge \bar{e}^2 \wedge\bar{\nu}^3\wedge \bar{\nu}^4) - u^2v^2\, \check{f}_2 \wedge \check{f}_3 \wedge (\bar{e}^1\wedge \bar{e}^2 - \bar{\nu}^3\wedge \bar{\nu}^4)\nonumber \\*
	&\quad- u^2v^2\, \check{f}_3 \wedge \check{f}_1 \wedge (\bar{e}^1\wedge \bar{\nu}^3 - \bar{\nu}^4\wedge \bar{e}^2) - u^2v^2\, \check{f}_1 \wedge \check{f}_2 \wedge (\bar{e}^1\wedge \bar{\nu}^4 - \bar{e}^2\wedge \bar{\nu}^3)\label{eq:*varphi}
\end{align}
\cite[(17), (18)]{CaliB}. 

Since $f^1$ is invariant under change of frame, it is globally defined on $L$. In fact, it can be written as $f^1= \vol_L - *_{X}\vol_L$. It thus spans a rank 1 bundle $E = \spn\{f^1\}$, whose orthogonal complement \smash{$F = E^\perp$}\vphantom{$\int^B$} in $\Lambda^2_{-}(T^*X)\vert_L$ is locally represented by \smash{$F \stackrel{\text{loc}}{=} \spn\{f^2, f^3\}$}.
The total spaces of these bundles are 3- and 4-dimensional submanifolds of $M^7=\Lambda^2_{-}(T^*\R^4)$, respectively. 

We examine the spaces 
\begin{align*}
	X_{\sigma}^E &= \bigl\{(\pointL, \eta + \sigma({\pointL})) \in \Lambda_{-}^2(T^*X^4)\vert_L \, \big| \, \pointL \in L,\, \eta \in E_{\pointL} \bigr\} \qquad (\text{\enquote{$E + \sigma$}}), \\ 
	X_{\eta}^F &= \bigl\{(\pointL, \eta({\pointL}) + \sigma) \in \Lambda_{-}^2(T^*X^4)\vert_L \, \big| \, \pointL \in L,\, \sigma \in F_{\pointL} \bigr\} \qquad (\text{\enquote{$\eta + F$}})
\end{align*}
for sections $\sigma \in \Gamma(F)$ and $\eta \in \Gamma(E)$. 
This is a \enquote{twisting} of the bundle $E$ ($F$) over $L$ obtained by affinely translating each fiber $E_{\pointL}$ ($F_{\pointL}$) by a vector $\sigma({\pointL}) \in F_{\pointL}$ ($\eta({\pointL}) \in E_{\pointL}$) in the orthogonal complement. 
Our goal is to determine necessary and sufficient conditions on $L,\, \sigma$ and $\eta$ so that $E + \sigma$ and $\eta + F$ are associative and coassociative in $M^7 = \Lambda^2_{-}(T^*X^4)$. 
We begin by establishing some needed formulas and constructing a holomorphic structure on $F$.

\begin{lemma}\label{lemma:nablaf}
	Let $\nabla$ denote the connection on $\Lambda^2_-(T^*X)\vert_L$ induced by the Levi-Civita connection on $X^4$. Using a normal frame \eqref{eq:defnormalcoords} at any $p^* \in L$ yields
	\begin{align*}
		\nabla_{e_i} f^1 &= (A_{i1}^4 - A_{i2}^3)f^2 + (-A^3_{i1}-A_{i2}^4)f^3, \\*
		\nabla_{e_i} f^2 &= (A_{i2}^3 - A_{i1}^4)f^1, \\*
		\nabla_{e_i} f^3 &= (A_{i2}^4 + A_{i1}^3)f^1
	\end{align*}
	at that point for $i=1,2$. 
\end{lemma}

\begin{proof}
	The proof is the same as for \cite[\Prop 4.1.1]{CaliA}.
\end{proof}

With $L$ being a $2$-dimensional oriented immersed submanifold of $X^4$, the tangent bundle \smash{$TL \stackrel{\text{loc}}{=}\spn\{e_1, e_2\}$}\vphantom{$\int^B$} and normal bundle \smash{$NL \stackrel{\text{loc}}{=}\spn\{\nu_3, \nu_4\}$} can be endowed with natural complex structures $J_T$ and $J_N$, locally given by $J_Te_1 = e_2, J_Te_2 = - e_1$ and $J_N\nu_3 = \nu_4,\allowbreak J_N \nu_4 = - \nu_3$. 
In fact, since $L$ is of dimension $2$, the Nijenhuis tensor of $J_T$ vanishes automatically, turning it into a complex structure on $L$.
As a consequence, it makes sense to talk about $L$ being \textbf{negative superminimal} in $X$, which means that $A^{J_N\nu} = - J _TA^\nu$ holds for all normal vector fields $\nu \in \Gamma(NL)$. 

Let $g_L$ denote the metric on $L$ induced by the metric on $X$. As $g_L$ is $J_T$-invariant, it defines a Hermitian metric on $(L,J_T)$. 
Furthermore, all $k$-forms on $L$ are trivial for $k > 3$, implying that its associated $2$-form $\omega_L$ must be closed. Hence, $(L, J_T, g_L, \omega_L)$ forms a complex $1$-dimensional Kähler manifold. 
Similarly, the rank $2$ vector bundle \smash{$F \stackrel{\text{loc}}{=} \spn\{f^2, f^3\}$} can be endowed with a complex structure $J_F$ locally given by $J_F f^2 = f^3$ and $J_Ff^3 = - f^2$. This turns $\{f^2\}$ into a local complex frame for $F$ and, consequently, $F$ into a rank $1$ complex vector bundle over $L$. 

The Levi-Civita connection on $X$ induces connections $\nabla$ on $\Lambda^2_-(T^*X)\vert_L$ and $\nabla^F = \pi_F \circ \nabla$ on $F$, where $\pi_F: \Lambda^2_-(T^*X)\vert_L = E \oplus F \rightarrow F$ denotes the projection onto $F$. The preceding lemma combined with the fact that $f^1 $ is orthogonal to $F$ then shows that \smash{$\nabla^F_{e_i} f^k = \pi_F (\nabla_{e_i} f^k)= 0$} for $i=1,2$ and $k=2,3$ at any $p^* \in L$, using a normal frame at that point. From this, we deduce that $(\nabla^F_{e_i} J_F)(f^k) = \nabla^F_{e_i}(J_Ff^k) - J_F(\nabla^F_{e_i}f^k) = 0$, which is equivalent to $\nabla^F J_F = 0$ as the latter is a tensor. 
Consequently, the connection $\nabla^F$ is complex linear with respect to $J_F$, and, by setting \vphantom{$\int^{B^B}$}\smash{$\nabla^F_{X + i Y} \sigma \stackrel{\text{def}}{=} \nabla^F_X \sigma + J_F(\nabla^F_Y \sigma)$} for $X,Y \in T_{p}L$, $p \in L$ and $\sigma \in \Gamma(F)$, it becomes a complex connection. This allows us to define the operator \smash{$\bar{\partial}_F \stackrel{\text{def}}{=} (\nabla^F)^{0,1}: \Omega^{r,s} (F) \rightarrow \Omega^{r,s+1} (F)$}\vphantom{$\int^B$}.
As $\bar{\partial}_F$ satisfies the Leibniz rule, it defines a pseudo-holomorphic structure on $F$. However, since $L$ has complex dimension $1$, there exist no $(0,2)$-forms on $F$.
Thus, $\bar{\partial}_F^2: \Omega^{r,s} (F) \rightarrow \Omega^{r,s+2} (F) = \{0\}$ must be trivial, implying that $\bar{\partial}_F = (\nabla^F)^{0,1}$ actually defines a holomorphic structure on $F$, turning it into a holomorphic line bundle. Then a section $\sigma \in \Gamma(F)$ is holomorphic if $\bar{\partial}_F \sigma = (\nabla^F)^{0,1} \sigma = 0$.
See \cite[\Ch 5]{Joyce2007} and \cite[\Ch 9]{Moroianu2007} for background on Kähler manifolds and holomorphic vector bundles. 

\begin{satz}\label{thm:assNEU} Let $X^4=S^4$ or $\mathbb{CP}^2$. Then: 
	\begin{enumerate}
		\item The submanifold $E + \sigma$ is associative in $\Lambda_{-}^2(T^*X^4)$ if and only if $L^2$ is minimal in $X^4$ and $\sigma \in \Gamma(F)$ is holomorphic. \label{item:assNEU1}
		\item The submanifold $\eta + F$ is coassociative in $\Lambda_{-}^2(T^*X^4)$ if and only if $L^2$ is negative superminimal in $X^4$ and $\eta \in \Gamma(E)$ is parallel with respect to the induced connection $\nabla^E$ on $E$ from the Levi-Civita connection on $X^4$. \label{item:assNEU2}
	\end{enumerate}
\end{satz}

\begin{rem}
	Note that the conditions on $L$, $\sigma$ and $\eta$ are the same as in the case of $\R^4$ in \cite{CaliA2}. 
\end{rem}

We split the proof into two parts for pedagogical reasons. 

\begin{proof}[Proof of \ref{item:assNEU1}] 
	We fix a point $(\pointL^*, \eta^* + \sigma({\pointL^*})) \in E + \sigma$, let $(e_1, e_2,  \nu_3, \nu_4)$ be a local orthonormal adapted frame along $L$ which is normal at $p^*$, and $f^1, f^2, f^3$ be the corresponding local anti-self-dual 2-forms. Moreover, we let $u = (u_1, u_2)$ be the coordinates for $L$ described in \autoref{sec:2ndfundform}, and $t_1$ be the coordinate on the fiber with respect to $f^1$. Then our fixed point has the coordinates $(0, t^*_1)$.
	
	The immersion of $E + \sigma$ into $M^7 = \Lambda^2_{-}(T^*X^4)$, $\Phi: X_{\sigma}^E \rightarrow M$, is locally given by 
	\begin{align*}
		\Phi: (u,t_1) \mapsto \bigl(x(u), t_1f^1(u) + \sigma(u)\bigr) = \bigl(x(u), t_1f^1(u) + a(u)f^2(u) + b(u)f^3(u)\bigr),
	\end{align*}
	where $x$ is the immersion of $L$ into $X^4$, and $a$ and $b$ are the coordinates of $\sigma$ with respect to the local trivialization \smash{$F \stackrel{\text{loc}}{=} \spn \{f^2, f^3\}$}. 
	Then the tangent space to $X_{\sigma}^E$ at $\omega^* = \Phi(0, t_1^*) \in M$ is spanned by $E_1= \Phi_*(\partial_{u_1}),\, E_2= \Phi_*(\partial_{u_2})$ and $F_1 = \Phi_*(\partial_{t_1})$, omitting the dependence on $(0,t^*)$.
	
	By \autoref{lemma:nablaf}, we have
	\begin{align*}
		\Ver_{\omega^*}(\nabla_{e_i} f^1) &= \Ver_{\omega^*}\bigl((A^4_{i1} - A^3_{i2})f^2 + (-A^3_{i1}-A^4_{i2})f^3\bigr) \\*
		&= (A^4_{i1} - A^3_{i2})\check{f}^2 + (-A^3_{i1}-A^4_{i2})\check{f}^3, \\
		\Ver_{\omega^*}(\nabla_{e_i} f^2) &= \Ver_{\omega^*}\bigl((A^3_{i2} - A^4_{i1})f^1\bigr) 
		= (A^3_{i2} - A^4_{i1})\check{f}^1, \\
		\Ver_{\omega^*}(\nabla_{e_i} f^3) &= \Ver_{\omega^*}\bigl((A^4_{i2} + A^3_{i1})f^1\bigr) 
		= (A^4_{i2} + A^3_{i1})\check{f}^1
	\end{align*}
	for $i=1,2$. 
	Using these formulas, we find that $E_i$ takes the form
	\begin{align*}
		E_i &= \bar{e}_i + \Ver_{\omega^*}\bigl(t_1^* \nabla_{e_i} f^1 + \nabla_{e_i}(af^2 + bf^3)\bigr) \\
		&= \bar{e}_i + t_1^* \Ver_{\omega^*}(\nabla_{e_i} f^1) + a \Ver_{\omega^*}(\nabla_{e_i} f^2)+ b \Ver_{\omega^*}(\nabla_{e_i} f^3) + a_i \check{f}^2 + b_i \check{f}^3\\*		
		&= \bar{e}_i + A_i \check{f}^1 + B_i \check{f}^2 + C_i \check{f}^3
	\end{align*}
	with 
	\begin{align*}
		A_i = a(A^3_{i2} - A^4_{i1}) + b(A^4_{i2} + A^3_{i1}), \quad
		B_i = t_1^* (A^4_{i1} - A^3_{i2})+ a_i, \quad
		C_i = t_1^*(-A^3_{i1}-A^4_{i2}) + b_i,
	\end{align*}
	and $a_i = \partial a/\partial u_i,\, b_i = \partial b/\partial u_i$ for $i=1,2$. 
	Lastly, the third basis vector is given by
	\begin{align*}
		F_1 &= \Phi_*\biggl(\frac{\partial}{\partial t_1}\biggr) = \check{f}^1. 
	\end{align*}
	
	Next, we determine when $E_2 \lrcorner E_1 \lrcorner F_1 \lrcorner \psi $ vanishes (\cf \ref{item:asscond}). 
	Using the formula
	\begin{align*}
		\psi &= u^4(\bar{e}^1\wedge \bar{e}^2 \wedge\bar{\nu}^3\wedge \bar{\nu}^4) - u^2v^2\, \check{f}_2 \wedge \check{f}_3 \wedge (\bar{e}^1\wedge \bar{e}^2 - \bar{\nu}^3\wedge \bar{\nu}^4)\nonumber \\*
		&\quad- u^2v^2\, \check{f}_3 \wedge \check{f}_1 \wedge (\bar{e}^1\wedge \bar{\nu}^3 - \bar{\nu}^4\wedge \bar{e}^2) - u^2v^2\, \check{f}_1 \wedge \check{f}_2 \wedge (\bar{e}^1\wedge \bar{\nu}^4 - \bar{e}^2\wedge \bar{\nu}^3)
	\end{align*}
	(see \eqref{eq:*varphi}), we compute
	\begin{align*}
		F_1 \lrcorner \psi &= - u^2v^2\bigl(-\check{f}_3 \wedge (\bar{e}^1\wedge \bar{\nu}^3 - \bar{\nu}^4\wedge \bar{e}^2) + \check{f}_2 \wedge (\bar{e}^1\wedge \bar{\nu}^4 - \bar{e}^2\wedge \bar{\nu}^3) \bigr), \\
		E_1 \lrcorner F_1 \lrcorner \psi 
		&= - u^2v^2(\check{f}_3 \wedge \bar{\nu}^3 - \check{f}_2 \wedge \bar{\nu}^4) \\
		&\quad - u^2v^2\bigl(-C_1 (\bar{e}^1\wedge \bar{\nu}^3 - \bar{\nu}^4\wedge \bar{e}^2) + B_1 (\bar{e}^1\wedge \bar{\nu}^4 - \bar{e}^2\wedge \bar{\nu}^3) \bigr), \\ 
		E_2 \lrcorner E_1 \lrcorner F_1 \lrcorner \psi 
		&= - u^2v^2(C_2\bar{\nu}^3 - B_2 \bar{\nu}^4 ) - u^2v^2(-C_1 \bar{\nu}^4 - B_1 \bar{\nu}^3 ) \\
		&= - u^2v^2\bigl((- B_1+ C_2)\bar{\nu}^3 + (-B_2-C_1)\bar{\nu}^4 \bigr).
	\end{align*}
	Since $u,v > 0$, $E_2 \lrcorner E_1 \lrcorner F_1 \lrcorner \psi$ equals zero if and only if the equations
	\begin{alignat}{9}
		&0 = - &B_1& + C_2\ &=&\ -& t_1^*& (A^4_{11} - A^3_{12}) - a_1 + t_1^*(-A^3_{12}-A^4_{22}) + b_2 \nonumber\\*
		&&&&=&\ - &t_1^*&(A^4_{11} + A^4_{22}) + \left(- a_1 + b_2 \right),\label{eq:asseqone} \\[0.5em]
		&0 = &B_2& + C_1 &=& &t_1^*& (A^4_{12} - A^3_{22}) + a_2 + t_1^*(-A^3_{11}-A^4_{12}) + b_1\nonumber\\*
		&&&&=&\ -& t_1^*&(A^3_{11} + A^3_{22}) + \left(a_2 + b_1 \right)\label{eq:asseqtwo}
	\end{alignat}
	are satisfied. 
	
	Given that $\eta^* \in E_{p^*}$ was arbitrary, the restricted tangent bundle $T X_\sigma^E \vert_{(X_\sigma^E)_{p^*}}$ has associative fibers if and only if \eqref{eq:asseqone} and \eqref{eq:asseqtwo} hold true at $p^*$ for all $t_1^* \in \R$. This is equivalent to the conditions
	\begin{enumerate}[label=(\Roman*)]
		\item $\Tr A^3= A^3_{11} + A^3_{22} = 0$ and $\Tr A^4= A^4_{11} + A^4_{22} = 0$ at $p^*$, or, in other words, $\Tr A \vert_{p^*}= 0$,
		\item $a_1 = b_2$ and $a_2 = - b_1$ at $p^*$. \label{condII} 
	\end{enumerate}
	
	As for \ref{condII}, recall that $\sigma$ is locally given by $\sigma =a f^2 + bf^3$ and that it is holomorphic if $\bar{\partial}_F \sigma = (\nabla^F)^{0,1} \sigma =0$. Since $e_1 + i e_2$ locally trivializes $(TL)^{0,1}$, this is equivalent to $\nabla^F_{e_1 + ie_2} \sigma =0$ .
	We compute
	\begin{align*}
		\nabla^F_{e_1 + i e_2} \sigma
		&= \nabla^F_{e_1} \sigma + J_F( \nabla^F_{e_2} \sigma) \\*
		&= \pi_F(\nabla_{e_1} \sigma ) + J_F\bigl(\pi_F(\nabla_{e_2} \sigma)\bigr) \\*
		&= \pi_F(a \nabla_{e_1} f^2 + b \nabla_{e_1}f^3 + a_1 f^2 + b_1 f^3)\\*
		&\quad+ J_F\bigl(\pi_F(a \nabla_{e_2} f^2 + b \nabla_{e_2}f^3 + a_2 f^2 + b_2 f^3)\bigr) \\*
		&= a_1 f^2 + b_1 f^3 + J_F(a_2 f^2 + b_2 f^3) \\*
		&= (a_1-b_2)f^2 + (a_2+b_1)f^3
	\end{align*}
	at $p^*$, where we used \autoref{lemma:nablaf} and the fact that $f^1$ is orthogonal to $F$. 
	Consequently, \ref{condII} is satisfied if and only if $\bar{\partial}_F \sigma$ vanishes at $p^*$.
	
	Since $p^* \in L$ was arbitrary, $ X_\sigma^E$ is associative in $M^7$, which is equivalent to \smash{$T X_\sigma^E \vert_{(X_\sigma^E)_{p}}$} having associative fibers for all $p \in L$, if and only if $L$ is minimal in $X$ and $\sigma$ is holomorphic.
\end{proof}

We proceed to the proof of the second statement.

\begin{proof}[Proof of \ref{item:assNEU2}] 
	We fix a point $(\pointL^*, \eta(p^*) + \sigma^*) \in \eta + F$, and let $f^1, f^2, f^3$ and $u = (u_1, u_2)$ be as before. Moreover, we let $t=(t_2,t_3)$ be the coordinates on the fiber with respect to the local trivialization \smash{$F \stackrel{\text{loc}}{=} \spn \{f^2, f^3\}$}. Then our fixed point has the coordinates $(0, t^*)$ for some $t^*= (t_2^*,t_3^*)$.
	
	The immersion of $\eta + F$ into $M$, $\Psi: X_{\eta}^F \rightarrow M$, is locally given by 
	\begin{align*}
		\Psi: (u, t) \mapsto \bigl(x(u), t_2f^3(u)+ t_3 f^3(u) + \eta(u)\bigr) = \bigl(x(u), t_2f^3(u)+ t_3 f^3(u) + \gamma(u)f^1(u)\bigr),
	\end{align*}
	where $x$ is the immersion of $L$ into $X^4$ and $\gamma \in C^\infty(L)$ is the globally defined function such that $\eta = \gamma f^1$. 
	Then the tangent space to $X_\eta$ at $\omega^* = \Psi(0, t^*) \in M$ is spanned by $E_1 = \Psi_*(\partial_{u_1}),\, E_2 = \Psi_*(\partial_{u_2}),\, F_2 = \Psi_*(\partial_{t_2})$ and $F_3 = \Psi_*(\partial_{t_3})$, omitting the dependence on $(0,t^*)$.
	
	As in the proof of the first statement, we use \autoref{lemma:nablaf} to compute $E_i$, and we find that 
	\begin{align*}
		E_i &= \bar{e}_i + \Ver_{\omega^*}\bigl(t_2^* \nabla_{e_i} f^2 + t_3^* \nabla_{e_i} f^3+ \nabla_{e_i}(\gamma f^1)\bigr)\\*
		&= \bar{e}_i + t_2^* \Ver_{\omega^*}(\nabla_{e_i} f^2) + t_3^* \Ver_{\omega^*}(\nabla_{e_i} f^3) + \gamma \Ver_{\omega^*}(\nabla_{e_i} f^1) + \gamma_i \check{f}^1 \\*
		&= \bar{e}_i + A_i\check{f}^1 + B_i \check{f}^2+ C_i \check{f}^3
	\end{align*}
	with
	\begin{align*}
		A_i = t_2^*(A_{i2}^3-A_{i1}^4) + t_3^*(A_{i2}^4+ A_{i1}^3) + \gamma_i, \quad
		B_i = \gamma(A_{i1}^4 -A_{i2}^3), \quad
		C_i = \gamma(-A_{i1}^3-A_{i2}^4),
	\end{align*}
	and $\gamma_i = \partial \gamma/\partial u_i$ for $i=1,2$. 
	Furthermore, $F_2$ and $F_3$ are given by
	\begin{align*}
		F_j &= \Psi_*\biggl(\frac{\partial}{\partial t_j} \biggr) = \check{f}^j, \qquad j=2, 3. 
	\end{align*}
	
	By \ref{item:coasscond}, $X_\eta^F$ is coassociative in $M$ if and only if $\varphi \vert_{X_\eta^F} =0$. Using the formula
	\begin{align*}
		\varphi &= v^3(\check{f}_1\wedge \check{f}_2 \wedge \check{f}_3) + u^2v\, \check{f}_1 \wedge (\bar{e}^1\wedge \bar{e}^2 - \bar{\nu}^3\wedge \bar{\nu}^4) \nonumber\\* 
		&\quad + u^2v\, \check{f}_2 \wedge (\bar{e}^1\wedge \bar{\nu}^3 - \bar{\nu}^4\wedge \bar{e}^2) + u^2v\, \check{f}_3 \wedge (\bar{e}^1\wedge \bar{\nu}^4 - \bar{e}^2\wedge \bar{\nu}^3)
	\end{align*}
	(see \eqref{eq:varphi}), we obtain $\varphi(F_2, F_3, \cdot) = v^3 \check{f}_1$, which implies $\varphi(F_2, F_3, E_i) = v^3 A_i$ for $ i=1,2$. 
	On the other hand, 
	\begin{align*}
		\varphi(E_1,E_2, \cdot) &= u^2v \check{f}_1 
		+ v^3\bigl((A_1B_2-A_2B_1)\check{f}_3 + (-A_1C_2+A_2C_1)\check{f}_2+(B_1C_2-B_2C_1)\check{f}_1 \bigr) \\*
		&\quad -u^2v(A_2 \bar{e}^2 + B_2\bar{\nu}^3+ C_2 \bar{\nu}^4) + u^2v(-A_1 \bar{e}^1 + B_1\bar{\nu}^4-C_1\bar{\nu}^3) \\
		&= \bigl(u^2v + v^3(B_1C_2 - B_2C_1)\bigr)\check{f}_1 + v^3(-A_1C_2+A_2C_1)\check{f}_2 + v^3(A_1B_2-A_2B_1)\check{f}_3 \\
		&\quad -u^2v \bigl(A_1 \bar{e}^1 + A_2 \bar{e}^2 + (B_2+C_1)\bar{\nu}^3 + (-B_1+C_2)\bar{\nu}^4\bigr)
	\end{align*}
	yields $\varphi(E_1,E_2 ,F_2) = v^3(-A_1C_2+A_2C_1)$ and $\varphi(E_1,E_2 ,F_3) = v^3(A_1B_2 - A_2B_1)$. 
	Since $v > 0$, we deduce that $\varphi $ vanishes on \smash{$T_{\omega^*}(\Psi(X_\eta^F))$} if and only if the conditions
	\begin{align*}
		A_1 = A_2 = 0, \quad -A_1C_2 + A_2C_1 = 0 \quad \text{ and }\quad A_1B_2-A_2B_1 =0
	\end{align*}
	are satisfied. As the last two follow from the first, this is equivalent to $A_i = 0$ for $ i=1,2$. 
	
	Since $\sigma^* \in F_{p^*}$ was arbitrary, the restricted tangent bundle \smash{$T X_\eta^F \vert_{(X_\eta^F)_{p^*}}$} has coassociative fibers if and only if $A_i(t) = t_2(A_{i2}^3-A_{i1}^4) + t_3(A_{i2}^4+ A_{i1}^3) + \gamma_i$ vanishes at $p^*$ for all $t \in \R^2$ and $i=1,2$.
	Define $\nu(t) = t_2 \nu_3 + t_3 \nu_4$ and $\nu^\perp(t) = J_N \nu(t) = - t_3 \nu_3 + t_2 \nu_4$.
	Then $A_i$ takes the form \smash{$A_i = A_{i2}^\nu - A_{i1}^{\nu^\perp} + \gamma_i$} for $i=1,2$.
	From the relations
	\begin{align*}
		A_i(\lambda t) &= A_{i2}^{\lambda\nu( t)} - A_{i1}^{\lambda\nu^\perp( t)} + \gamma_i = \lambda(A_{i2}^{\nu( t)} - A_{i1}^{\nu^\perp( t)} ) + \gamma_i, \\*
		A_i(\lambda t^\perp) &= A_{i2}^{\lambda\nu( t^\perp)} - A_{i1}^{\lambda\nu^\perp( t^\perp)} + \gamma_i = \lambda(A_{i2}^{\nu^\perp( t)} + A_{i1}^{\nu( t)} ) + \gamma_i
	\end{align*}
	for $(t_2, t_3)^\perp = (-t_3, t_2)$ and $\lambda \in \R$, we deduce that $A_i$ vanishes 
	for all $t$
	if and only if
	\begin{align*}
		A_{i2}^{\nu} - A_{i1}^{\nu^\perp} = 0, \quad A_{i2}^{\nu^\perp} + A_{i1}^{\nu} = 0 \quad \text{ and } \quad \gamma_i= 0.
	\end{align*}
	Combining these conditions yields
	\begin{align*}
		A^{J_N\nu} = A^{\nu^\perp} = 
		\begin{pmatrix}
			A_{11}^{\nu^\perp} & A_{12}^{\nu^\perp} \\
			A_{12}^{\nu^\perp} & A_{22}^{\nu^\perp}
		\end{pmatrix}
		= 
		\begin{pmatrix}
			A_{12}^{\nu} & A_{22}^{\nu} \\
			-A_{11}^{\nu} & -A_{12}^{\nu}
		\end{pmatrix}
		= 
		\begin{pmatrix}
			0 & 1 \\
			-1 & 0
		\end{pmatrix}
		\begin{pmatrix}
			A_{11}^{\nu} & A_{12}^{\nu} \\
			A_{12}^{\nu} & A_{22}^{\nu}
		\end{pmatrix}
		= - J_T A^{\nu}
	\end{align*}
	and $\gamma_1 = \gamma_2 = 0$. 
	Since $\nu(t)\vert_{p^*}$ and $\nu^\perp(t)\vert_{p^*}$ form a basis for $N_{p^*}L$ whenever $t \neq 0$, this shows that \smash{$T X_\eta^F \vert_{(X_\eta^F)_{p^*}}$} has coassociative fibers if and only if $A^{J_N\nu} = - J_T A^{\nu}$ holds for all $\nu \in N_{p^*}L$, and $\gamma$ satisfies $\gamma_1(p^*) = \gamma_2(p^*) = 0$.
	
	From \autoref{lemma:nablaf}, we know that $(\nabla_{e_i} f ^1)\vert_{p^*} $ is orthogonal to $E$. Using this fact, we compute
	\begin{align*}
		\nabla^E_{e_i} \eta = \pi_E(\nabla_{e_i} \eta) = \pi_E(\gamma_i f^1 + \gamma \nabla_{e_i} f^1) = \gamma_i f^1, \qquad i=1,2,
	\end{align*}
	at $p^*$, where $\pi_E: E \oplus F \rightarrow E$ denotes the projection onto $E$. This shows that the condition $\gamma_1(p^*) = \gamma_2(p^*) = 0$ is equivalent to $(\nabla^E \eta) \vert_{p^*} = 0$.
	
	Given that $p^*$ was arbitrary, \smash{$X_\eta^F$} is coassociative in $M$, which is equivalent to \smash{$T X_\eta^F \vert_{(X_\eta^F)_{p}}$} having coassociative fibers for all $p \in L$, if and only if $L$ is negative superminimal and $\eta$ is parallel with respect to $\nabla^E$.
\end{proof}

\section{Cayley submanifolds of \texorpdfstring{$\negspinS(S^4)$}{} with the Bryant--Salamon metric}\label{sec:CayleyNew}

We now come to the construction of Cayley submanifolds in the negative spinor bundle $\negspinS(S^4)$ of the $4$-dimensional sphere. 
To begin with, let us briefly review some preliminaries discussed in \cite[\Sec 4.4]{CaliA} for $\R^4$ with the standard metric, which also hold true for $S^4$ with the standard round metric $g$.

Let $e^1, \dots, e^4$ be an orthonormal basis of $T_{x}^*S^4$ at a fixed point $x \in S^4$. Then the Clifford algebra $\Cl(T_{x}^*S^4) \cong \Cl(T_{x}S^4)$ is generated by $e^1, \dots, e^4$ subject to the relations 
\begin{align}\label{eq:Cliffordrel}
	e^i \cdot e^j + e^j \cdot e^i = -2 \delta^{ij}.
\end{align}
We write $\pmspinS$ for the $\pm 1$ eigenspace of the pinor representation $\gamma(\lambda) \in \End(\spinS)$ of the volume element $\lambda= e^1 \cdot e^2 \cdot e^3 \cdot e^4 \in \Cl(T_{x}^*S^4)$. Both $\posspinS$ and $\negspinS$ are isomorphic to the quaternions $\Hbb$, and Clifford multiplication by a covector $\alpha \in T_{x}^*S^4$ interchanges them since $\lambda \cdot \alpha = - \alpha \cdot \lambda$.
On the other hand, octonionic multiplication satisfies $u(uv)= u^2v$ and $u_1(\bar{u}_2v) = -u_2(\bar{u}_1v)$ for all $u, u_1, u_2, v \in \Oct$ with $u_1$ and $u_2$ orthogonal. 
Combining these two identities yields
\begin{align}\label{eq:octonionicrel}
	u_i(u_jv) + u_j(u_iv) = -2\delta_{ij}v
\end{align}
for any orthonormal basis $u_1, \dots, u_4$ of $\Hbb \mathrm{e}$ and any $v \in \Oct$.
Furthermore, multiplication by elements in $\Hbb \mathrm{e}$ interchanges $\Hbb$ and $\Hbb \mathrm{e}$.
Comparing \eqref{eq:Cliffordrel} and \eqref{eq:octonionicrel}, we see that the pinor representation at each point $x \in S^4$ is obtained from octonionic multiplication by identifying $\spinS = \posspinS \oplus \negspinS \cong \Hbb \mathrm{e} \oplus \Hbb \cong \Oct$ and $T_{x}^*S^4 \cong \Hbb \mathrm{e}$. 
For covectors $\alpha \in T_{x}^*S^4$, we write it as 
\begin{align*}
	\gamma: T_{x}^*S^4 \rightarrow \End(\posspinS \oplus \negspinS), \quad \gamma(\alpha)(s) = \alpha s,
\end{align*}
where the product denotes octonionic multiplication. When composing two elements of this representation, it is crucial to remember that $(\gamma(\alpha_1)\gamma(\alpha_2))(s)\allowbreak = \alpha_1(\alpha_2s)$ is in general not equal to $(\alpha_1\alpha_2)s$ because $\Oct$ is not associative. For more details on spin geometry and representations, see, \eg \cite[\Ch 9--11]{Harvey1990}, \cite[\Sec 50]{Wendl}, \cite[\Sec B]{Merkel2024}.

Next, we endow $M^8 = \negspinS(S^4)$ with a parallel $\Spin(7)$-structure following \cite{BryantSalamon1989}. 
Recall that there exists a natural connection $\nabla$, the spin connection, on $M^8$, induced by the Levi-Civita connection on $(S^4,g)$. 
This provides a natural splitting of its tangent space $T_{s} M \cong \mathcal{H}_{s} \oplus \mathcal{V}_{s}$ into the horizontal and vertical spaces for any $s \in M$.
We can identify $\mathcal{H}_{s}$ with the tangent space $T_{\pi(s)}S^4$ of $S^4$ via the linear isomorphism $\Hor_{s} = (\pi_*\vert_{\mathcal{H}_{s}})^{-1}: T_{\pi(s)}S^4 \rightarrow \mathcal{H}_{s}$, where $\pi: M^8 = \negspinS(S^4) \rightarrow S^4$ denotes the projection onto the base.
On the other hand, $\mathcal{V}_{s}$ can be identified with the fiber \smash{$M_{\pi(s)} = (\negspinS)_{\pi(s)}(S^4)$} through the linear isomorphism \smash{$\Ver_{s}: M_{\pi(s)} \rightarrow \mathcal{V}_{s} = T_s(M_{\pi(s)})$}, {$ \sigma \mapsto \frac{d}{dt} (s +t \sigma)\vert_{t=0}$}.
Due to these identifications, the metric $g$ on $S^4$ induces metrics $g_{\mathcal{H}}$ and $g_{\mathcal{V}}$ with natural volume forms $\vol_{\mathcal{H}}$ and $\vol_{\mathcal{V}}$ on $\mathcal{H}$ and $\mathcal{V}$, respectively. 

\begin{satz}[{\cite[\Thm 4.2]{BryantSalamon1989}, \cite[\Subsec 2.2]{Karigiannis2008}}]\label{BSV}
	Consider $S^4$ with the standard round metric and let $r$ denote the radial coordinate in the vertical fibers. Then there exist positive functions $u= u(r)$ and $v=v(r)$ such that 
	\begin{align*}
		g_{M^8} = u^2 g_{\mathcal{H}} \oplus v^2 g_{\mathcal{V}}
	\end{align*}
	defines a complete metric on $M^8 = \negspinS(S^4)$ with holonomy equal to $\Spin(7)$. Its fundamental $4$-form $\Phi$ is given by
	\begin{align*}
		\Phi = u^4\, \vol_{\mathcal{H}} - u^2 v^2( \omega_1 \wedge \sigma^1 + \omega_2\wedge \sigma^2 + \omega_3 \wedge \sigma^3) + v^4\, \vol_{\mathcal{V}},
	\end{align*}
	where $\omega_1, \omega_2, \omega_3$ is an orthogonal basis of norm $\sqrt{2}$ for the self-dual $2$-forms on $\mathcal{H}$ and $\sigma^1,\allowbreak \sigma^2,\allowbreak \sigma^3$ is the corresponding orthogonal basis for the self-dual $2$-forms on $\mathcal{V}$. 
\end{satz}

\begin{rem}
	The factor $\sqrt{2}$ in the preceding theorem appears due to our convention for the inner product on the exterior algebra $\Lambda^kV$ over some vector space $V$, namely $\langle v_1 \wedge \cdots \wedge v_k , w_1 \wedge \cdots \wedge w_k \rangle = \det(\langle v_i, w_j \rangle)$ for $v_i, w_j \in V$.
\end{rem}

Let us now restrict the vector bundle $M^8 = \negspinS(S^4)\rightarrow S^4$ to an oriented immersed submanifold $L^2 \subset S^4$ and fix some oriented local orthonormal adapted frame $(e_1, e_2, \nu_3, \nu_4)$ along $L$ with dual coframe \smash{$(e^1, e^2, \nu^3, \nu^4)$}.
We denote its horizontal lift to $\mathcal{H}$ by $\bar{e}_i = \Hor\,e_i$, $\bar{\nu}_j = \Hor\,\nu_j$, and their dual horizontal $1$-forms by $\bar{e}^i, \bar{\nu}^j,\, i=1,2, j=3, 4$. 

Since $e^1, e^2, \nu^3, \nu^4$ are orthonormal, they satisfy $e^1 \cdot e^2 = \frac{1}{2} e^1 \wedge e^2$ and $\nu^3 \cdot \nu^4 = \frac{1}{2} \nu^3 \wedge \nu^4$ \cite[(B.1)]{Merkel2024}, which implies that the terms $\gamma(e^1)\gamma(e^2) = \gamma(e^1 \cdot e^2)$ and $\gamma(\nu^3)\gamma(\nu^4) =$ $\gamma(\nu^3 \cdot \nu^4)$ are independent of the choice of frame and hence globally defined. 
Let us now focus on a fixed point $\pointL \in L$ and consider the restricted operators $\gamma(e^1)\gamma(e^2)$ and \smash{$\gamma(\nu^3)\gamma(\nu^4): \negspinS \rightarrow \negspinS$}.
The identity $(u\mathrm{e})((v\mathrm{e})(wy)) = w((u\mathrm{e})((v\mathrm{e})y))$ for $u,v,w,y \in \Hbb$ shows that both $\gamma(e^1)\gamma(e^2)$ and $\gamma(\nu^3)\gamma(\nu^4)$ are complex linear with respect to the natural complex structure $j_L = e^1e^2 \in \{u \in \im\, \Hbb \mid |u| = 1 \}$ on $\negspinS \cong \Hbb$. 
(A quick computation yields that $j_L$ is independent of the choice of frame.)
As $\gamma(e^1 \cdot e^2)^2 = \gamma(\nu^3 \cdot \nu^4)^2 = -1$, the two operators share the eigenvalues $j_L$ and $-j_L$. Combining this with the fact that they are simultaneously diagonalizable because they commute, $\gamma(e^1)\gamma(e^2)$ and $\gamma(\nu^3)\gamma(\nu^4)$ differ at most by a sign. 
Due to the relation $\gamma(e^1 \cdot e^2) \gamma(\nu^3 \cdot \nu^4) = \gamma(\lambda) = \pm1$ on $\pmspinS$, they must be equal on $\negspinS$.
Consequently, there exists a global canonical complex structure $\Gamma$ on the restricted bundle $\negspinS(S^4)\vert_L $, locally given by $\Gamma = \gamma(e^1)\gamma(e^2) = \gamma(\nu^3)\gamma(\nu^4): \negspinS \rightarrow \negspinS$. 
This operator provides a splitting of $\negspinS(S^4)\vert_L = V_+ \oplus V_-$ into the two eigenbundles $ V_\pm$ of rank $2$ corresponding to its eigenvalues $\pm j_L$.
The total spaces of these bundles are $4$-dimensional submanifolds of $\negspinS(S^4)$. 

Using similar reasoning as in the derivation of the equality $\gamma(e^1)\gamma(e^2) = \gamma(\nu^3)\gamma( \nu^4)$, we obtain
\begin{align}\label{eq:gammanegSpin}
	\gamma(e^1 \wedge e^2) = \gamma(\nu^3 \wedge \nu^4), \quad
	\gamma(e^1 \wedge \nu^3) = - \gamma(e^2 \wedge \nu^4), \quad 
	\gamma(e^1 \wedge \nu^4) = \gamma(e^2 \wedge \nu^3). 
\end{align}
Consider the standard basis 
\begin{align*}
	f^1 = e^1 \wedge e^2 + \nu^3 \wedge \nu^4, \quad f^2 = e^1 \wedge \nu^3 + \nu^4 \wedge e^2, \quad f^3 = e^1 \wedge \nu^4 + e^2 \wedge \nu^3
\end{align*}
of self-dual $2$-forms on $S^4$. 
By \eqref{eq:gammanegSpin}, it satisfies
\begin{align}
	\gamma(f^1) &= \gamma(e^1 \wedge e^2) + \gamma(\nu^3 \wedge \nu^4) = 2\gamma(e^1 \wedge e^2) = 2 \gamma(\nu^3 \wedge \nu^4)= 4 \Gamma,\nonumber \\*
	\gamma(f^2) &= \gamma(e^1 \wedge \nu^3) + \gamma(\nu^4 \wedge e^2) = 2 \gamma(e^1 \wedge \nu^3) = - 2 \gamma(e^2 \wedge \nu^4),\label{eq:gammafk} \\*
	\gamma(f^3) &= \gamma(e^1 \wedge \nu^4) + \gamma(e^2 \wedge \nu^3) = 2\gamma(e^1 \wedge \nu^4) = 2 \gamma(e^2 \wedge \nu^3). \nonumber 
\end{align}
Using this, we compute that $\gamma(f^i)\gamma(f^j) = - \gamma(f^j)\gamma(f^i)$ for $i \neq j$, $\gamma(f^i)^2 = -16$ and 
\begin{align}\label{eq:gammafikombi}
	\gamma(f^1) \gamma(f^2) &= 4 \gamma(e^1\wedge e^2)\gamma(e^1\wedge \nu^3) = 8 \gamma(e^2\wedge \nu^3) = 4 \gamma(f^3),\nonumber \\
	\gamma(f^1) \gamma(f^3) &= 4 \gamma(e^1\wedge e^2)\gamma(e^1\wedge\nu^4) = 8 \gamma(e^2 \wedge\nu^4) = -4 \gamma(f^2), \\
	\gamma(f^2) \gamma(f^3) &= 4 \gamma(e^1\wedge\nu^3)\gamma(e^1\wedge\nu^4) = 8 \gamma(\nu^3 \wedge \nu^4) = 4 \gamma(f^1). \nonumber
\end{align}

Now fix a local unit spinor $s_1$ in $V_+$. 
Then {$\{s_1, s_2 = \frac{1}{4}\gamma(f^1) s_1 = \Gamma s_1 = j_L s_1\}$} and {$\{s_3 = \frac{1}{4}\gamma(f^2) s_1, s_4 = \frac{1}{4}\gamma(f^3) s_1 = \Gamma s_3 = - j_L s_3 \}$} form local orthonormal frames for $V_+$ and $V_-$, respectively. 
The latter follows from the fact that $\Gamma$ anti-commutes with both $\gamma(f^2)$ and $\gamma(f^3)$, and that $\Gamma \gamma(f^2) = \gamma(f^3)$. We denote the vertical lifts of these spinors to $\mathcal{V}$ by $\check{s}_k$, with dual vertical $1$-forms $\check{s}^k,\, k = 1, \dots, 4$. 
As in the $\Gtwo$ case, we summarize the described setup using a diagram: 
\begin{equation*}
	\begin{tikzcd}[column sep=-0.5em, row sep=1em]
		& \spn \{\bar{e}_1, \bar{e}_2, \bar{\nu}_3, \bar{\nu}_4\}\vert_{s} \arrow[d, equal, shorten=1mm] && \spn \{\check{s}_1, \check{s}_2, \check{s}_3, \check{s}_4\}\vert_{s} \arrow[d,equal, shorten=1mm]\\
		T_{s} M = T_{s}\bigl(\negspinS(S^4)\bigr) \cong &\mathcal{H}_{s} &\oplus& \mathcal{V}_{s} \\
		&&&\\
		&T_{\pi(s)}S^4 \arrow[uu, "\Hor_{s}", "\cong"'] && M_{\pi(s)} \arrow[uu, "\Ver_{s}", "\cong"'] = (\negspinS)_{\pi(s)}(S^4) \\
		&\spn \{e_1, e_2, \nu_3, \nu_4\}\vert_{\pi(s)} \arrow[u, equal, shorten=0.5mm] && \spn \{s_1, s_2, s_3, s_4\}\vert_{\pi(s)} \arrow[u,equal, shorten=0.5mm]
	\end{tikzcd}
\end{equation*}
Consequently, the fundamental $4$-form $\Phi$ in \autoref{BSV} restricted to $L$ is locally given by
\begin{align}\label{eq:PhiCayley}
	\Phi &= u^4 \bar{e}^1\bar{e}^2\bar{\nu}^3\bar{\nu}^4 - u^2 v^2(\omega_1\sigma^1 + \omega_2 \sigma^2 + \omega_3\sigma^3) + v^4\check{s}^1\check{s}^2\check{s}^3\check{s}^4 \nonumber\\*
	&= u^4 \bar{e}^1\bar{e}^2\bar{\nu}^3\bar{\nu}^4 
	- u^2v^2(\bar{e}^1\bar{e}^2 + \bar{\nu}^3\bar{\nu}^4)(\check{s}^1\check{s}^2 + \check{s}^3\check{s}^4) 
	- u^2v^2(\bar{e}^1\bar{\nu}^3 + \bar{\nu}^4\bar{e}^2)(\check{s}^1\check{s}^3 + \check{s}^4\check{s}^2)\nonumber\\*
	&\quad - u^2v^2(\bar{e}^1\bar{\nu}^4 + \bar{e}^2\bar{\nu}^3)(\check{s}^1\check{s}^4 + \check{s}^2\check{s}^3) 
	+ v^4 \check{s}^1\check{s}^2\check{s}^3\check{s}^4,
\end{align}
where we omitted the wedge product symbols for clarity.

\begin{rem}
	Contrary to \cite{CaliA, CaliA2} and this paper, the authors of \cite{CaliB} used the sign convention $\lambda = -e^1\cdot e^2\cdot \nu^3 \cdot \nu^4$ for the volume element. 
	Due to this, they actually worked in the positive spinor bundle. 
	The general idea remains the same but some adaptions are necessary to work out the statements and proof for the negative spinor bundle. In particular, the fundamental $4$-form $\Phi_+$ on $\posspinS(S^4)$ differs by some signs from our formula for $\Phi$ on $\negspinS(S^4)$ (see \eqref{eq:PhiCayley}):
	\begin{align*}
		\Phi_+
		&= u^4 \bar{e}^1\bar{e}^2\bar{\nu}^3\bar{\nu}^4 
		+ u^2v^2(\bar{e}^1\bar{e}^2 - \bar{\nu}^3\bar{\nu}^4)(\check{s}^1\check{s}^2 - \check{s}^3\check{s}^4) 
		+ u^2v^2(\bar{e}^1\bar{\nu}^3 - \bar{\nu}^4\bar{e}^2)(\check{s}^1\check{s}^3 - \check{s}^4\check{s}^2)\nonumber\\
		&\quad + u^2v^2(\bar{e}^1\bar{\nu}^4 - \bar{e}^2\bar{\nu}^3)(\check{s}^1\check{s}^4 - \check{s}^2\check{s}^3) 
		+ v^4 \check{s}^1\check{s}^2\check{s}^3\check{s}^4
	\end{align*}
	\cite[\Subsec 2.2]{Karigiannis2008}, \cite[\Thm 4.5]{CaliB}.
	For the purpose of being coherent and sticking to one convention, we keep considering the negative spinor bundle, for which their result can be proved analogously to \cite[\Thm 4.8]{CaliB}.
\end{rem}

We examine the space 
\begin{align*}
	X_{\psi} = \bigl\{(\pointL, \xi + \psi({\pointL})) \in \negspinS(S^4)\vert_L \, \big| \, \pointL \in L,\, \xi \in (V_+)_{\pointL} \bigr\} \qquad(\text{\enquote{$V_+ + \psi$}})
\end{align*}
for a section $\psi \in \Gamma(V_-)$. 
This is a \enquote{twisting} of the bundle $V_+$ over $L$ obtained by affinely translating each fiber $(V_+)_{\pointL}$ by a vector $\psi({\pointL}) \in (V_-)_{\pointL}$ in the orthogonal complement. 
Our goal is to determine necessary and sufficient conditions on $L$ and $\psi$ for $V_+ + \psi$ to be Cayley in $\negspinS(S^4)$.
We start with a lemma and the construction of a holomorphic structure on $V_\pm$. 

\begin{lemma}\label{lemma:nablaspinor}
	Let $\nabla$ denote the connection on $\negspinS(S^4)\vert_L$ induced by the Levi-Civita connection on $S^4$. Then $\nabla_{e_i} \Gamma$ interchanges $V_+$ and $V_-$, and every local section $s$ of $V_\pm$ satisfies
	\begin{align*}
		\nabla_{e_i} s= \mp \frac{1}{2}j_L (\nabla_{e_i} \Gamma)s
	\end{align*}
	for $i=1,2$. 
\end{lemma}

\begin{proof}
	Cf.~\cite[\Subsec 4.2]{CaliB}. 
\end{proof}

As in the previous subsection, $L$ can be viewed as a complex 1-dimensional Kähler manifold. 
Additionally, we equip $V_\pm$ with the complex structure $J_\pm = - \Gamma$. 
That is, $J_+s_1 = - s_2, J_+s_2 = s_1$ on $V_+$ and $J_-s_3 = - s_4, J_-s_4 = s_3$ on $V_-$. 
This turns both $V_+$ and $V_-$ into complex line bundles over $L$. 
(Without the minus sign, the following theorem would require $\psi$ to be anti-holomorphic rather than holomorphic.)

The Levi-Civita connection on $S^4$ induces connections $\nabla$ on $\negspinS(S^4)\vert_L$ and $\nabla^{V_\pm} = \pi_{V_\pm} \circ \nabla$ on $V_\pm$, where $\pi_{V_\pm}: V_+ \oplus V_- \rightarrow V_\pm$ denotes the projection onto $V_\pm$. 
Let $s$ be a local section of $V_\pm$ and $i \in \{1,2\}$. According to \autoref{lemma:nablaspinor}, $\nabla_{e_i} s$ is orthogonal to $V_\pm$, which implies that $\nabla_{e_i}^{V_\pm} s = \pi_{V_\pm}(\nabla_{e_i} s) = 0$. 
From this, we obtain $(\nabla^{V_\pm}_{e_i} J_\pm)(s) = \nabla^{V_\pm}_{e_i} (J_\pm s) - J_\pm(\nabla^{V_\pm}_{e_i} s) = 0$, showing that $J_\pm$ is parallel with respect to $\nabla^{V_\pm}$. 
Thus, the connection $\nabla^{V_\pm}$ is complex linear with respect to $J_\pm$, 
and, by setting \vphantom{$\int^B$}\smash{$\nabla^{V_\pm}_{X + i Y} s \stackrel{\text{def}}{=} \nabla^{V_\pm}_X s + J_{\pm}(\nabla^{V_\pm}_Y s)$}\vphantom{$\int^B$} for $X,Y \in T_{p}L$ and $p \in L$, it becomes a complex connection.
Following the same reasoning as in the previous subsection, we conclude that $V_\pm$ is a holomorphic line bundle with holomorphic structure given by $\bar{\partial}_{V_\pm} = (\nabla^{V_\pm})^{0,1}$.

\begin{satz}\label{thm:cayleyNEU}
	The submanifold $V_+ + \psi$ is Cayley in $\negspinS(S^4)$ if and only if $L^2$ is minimal in $S^4$ and $\psi \in \Gamma(V_-)$ is holomorphic. 
\end{satz}

\begin{rem}
	Note that the conditions on $L$ and $\psi$ are the same as in the case of $\R^4$ in \cite{CaliA2}.
\end{rem}

\begin{proof}
	We fix a point $(\pointL^*, \xi^* + \psi({\pointL^*})) \in V_+ + \psi$, let $(e_1, e_2,  \nu_3, \nu_4)$ be a local orthonormal adapted frame along $L$ which is normal at $p^*$, and suppose $f^1, f^2, f^3$ are the corresponding local self-dual 2-forms and $s_1, \dots, s_4$ are the corresponding local spinors defined above. 
	Moreover, we let $u = (u_1, u_2)$ be the coordinates for $L$ described in \autoref{sec:2ndfundform}, and $t=(t_1,t_2)$ be the coordinates on the fiber with respect to the local trivialization \smash{$V_+ \stackrel{\text{loc}}{=} \spn\{s_1, s_2\}$}\vphantom{$\int^B$}. Then our fixed point has the coordinates $(0, t^*)$ for some $t^*=(t_1^*,t_2^*)$.
	
	The immersion of $V_+ + \psi$ into $M^8 = \negspinS(S^4)$, $\Psi: X_\psi \rightarrow M$, is locally given by 
	\begin{align*}
		\Psi: (u, t) \mapsto\ & \bigl(x(u), t_1s_1(u) + t_2s_2(u)+ \psi(u)\bigr)\\*
		& = \bigl(x(u), t_1s_1(u) + t_2s_2(u)+ a(u)s_3(u) + b(u)s_4(u)\bigr), 
	\end{align*}
	where $x$ is the immersion of $L$ into $X^4$, and $a$ and $b$ are the coordinates of $\psi$ with respect to the local trivialization \smash{$V_- \stackrel{\text{loc}}{=} \spn\{s_3, s_4\}$}. 
	Then the tangent space to $X_\psi$ at $s^* = \Psi(0, t^*) \in M$ is spanned by $E_i = \Psi_*(\partial_{u_i})$ and $F_j = \Psi_*(\partial_{t_j})$ for $i,j=1,2$, omitting the dependence on $(0,t^*)$.
	
	By \autoref{lemma:nablaspinor}, we have $\nabla_{e_i} s_k= - \frac{1}{2}j_L (\nabla_{e_i} \Gamma)s_k$ and $\nabla_{e_i} s_l= \frac{1}{2}j_L (\nabla_{e_i} \Gamma)s_l$ for $k=1,2$ and $l=3,4$. 
	As we are working with a normal frame, we can use \eqref{eq:normalcoord} and \eqref{eq:gammafk} to compute
	\begin{align*}
		\nabla_{e_i} \Gamma 
		&= \gamma(\nabla_{e_i} e^1)\gamma(e^2) + \gamma(e^1)\gamma(\nabla_{e_i}e^2) \\*
		&= \gamma(-A_{i1}^3\nu^3 - A_{i1}^4 \nu^4)\gamma(e^2) + \gamma(e^1)\gamma(-A_{i2}^3\nu^3 - A_{i2}^4 \nu^4) \\
		&= \frac{1}{2} \bigl(-A_{i1}^3\gamma(\nu^3 \wedge e^2) -A_{i1}^4\gamma(\nu^4 \wedge e^2) -A_{i2}^3\gamma(e^1\wedge\nu^3) -A_{i2}^4\gamma(e^1\wedge\nu^4) \bigr)\\
		&= \frac{1}{4}\bigl((-A_{i2}^3 - A_{i1}^4)\gamma(f^2) + (A_{i1}^3 - A_{i2}^4)\gamma(f^3)\bigr)\\
		&= B_i\frac{1}{4}\gamma(f^2) + C_i\frac{1}{4}\gamma(f^3),
	\end{align*}
	where we set $B_i = -A_{i2}^3 - A_{i1}^4$ and $C_i = A_{i1}^3 - A_{i2}^4$.
	Substituting the expressions {$s_2 = \frac{1}{4}\gamma(f^1)s_1$}, $ s_3 = \frac{1}{4}\gamma(f^2)s_1$ and $s_4 = \frac{1}{4}\gamma(f^3)s_1$ into $(\nabla_{e_i} \Gamma)s_k$ and then applying \eqref{eq:gammafikombi} yields
	\begin{align*}
		(\nabla_{e_i} \Gamma)s_1 &= B_i\frac{1}{4}\gamma(f^2)s_1 + C_i\frac{1}{4}\gamma(f^3)s_1
		= B_is_3 + C_i s_4, \\
		(\nabla_{e_i} \Gamma)s_2 &= B_i\frac{1}{16}\gamma(f^2)\gamma(f^1)s_1 + C_i\frac{1}{16}\gamma(f^3)\gamma(f^1)s_1
		= - B_i\frac{1}{4}\gamma(f^3)s_1 + C_i\frac{1}{4}\gamma(f^2)s_1\\*
		&= C_i s_3 -B_is_4, \\
		(\nabla_{e_i} \Gamma)s_3 &= B_i\frac{1}{16}\gamma(f^2)\gamma(f^2)s_1 + C_i\frac{1}{16}\gamma(f^3)\gamma(f^2)s_1
		= - B_i s_1 - C_i\frac{1}{4}\gamma(f^1)s_1\\
		&= - B_i s_1 - C_i s_2, \\
		(\nabla_{e_i} \Gamma)s_4 &= B_i\frac{1}{16}\gamma(f^2)\gamma(f^3)s_1 + C_i\frac{1}{16}\gamma(f^3)\gamma(f^3)s_1
		= B_i\frac{1}{4}\gamma(f^1)s_1- C_i s_1\\
		&= - C_i s_1 + B_i s_2.
	\end{align*}
	
	Using these formulas, we find that $E_i$ takes the form 
	\begin{align*}
		E_i & = \bar{e}_i + \Ver_{s^*}\bigl(t_1^* \nabla_{e_i} s_1 + t_2^* \nabla_{e_i} s_2+ \nabla_{e_i}(as_3 + bs_4)\bigr)\\
		&= \bar{e}_i + \Ver_{s^*}(a_is_3 + b_i s_4) + \Ver_{s^*}(t_1^* \nabla_{e_i} s_1 + t_2^* \nabla_{e_i} s_2 + a \nabla_{e_i} s_3 + b \nabla_{e_i} s_4) \\
		&= \bar{e}_i + a_i\check{s}_3 + b_i \check{s}_4 \\
		&\quad- \frac{1}{2}j_L \Ver_{s^*}\bigl(t_1^* (\nabla_{e_i} \Gamma) s_1 + t_2^* (\nabla_{e_i} \Gamma) s_2\bigr) +\frac{1}{2}j_L \Ver_{s^*}\bigl(a(\nabla_{e_i} \Gamma) s_3 + b (\nabla_{e_i} \Gamma) s_4\bigr)\\
		&= \bar{e}_i + a_i\check{s}_3 + b_i \check{s}_4 \\*
		&\quad- \frac{1}{2}j_L \Ver_{s^*}\bigl((t_1^*B_i+ t_2^* C_i)s_3 + (t_1^*C_i- t_2^* B_i) s_4\bigr)\\*
		&\quad+\frac{1}{2}j_L \Ver_{s^*}\bigl((-a B_i - b C_i)s_1 +(- aC_i + b B_i)s_2\bigr)\\
		&= \bar{e}_i + a_i\check{s}_3 + b_i \check{s}_4 \\
		&\quad- \frac{1}{2}\Ver_{s^*}\bigl( (t_1^*C_i- t_2^* B_i) s_3 -(t_1^*B_i+ t_2^* C_i)s_4\bigr)\\
		&\quad+\frac{1}{2} \Ver_{s^*}\bigl( -(- aC_i + b B_i)s_1+(-a B_i - b C_i)s_2\bigr)\\
		&= \bar{e}_i + a_i\check{s}_3 + b_i \check{s}_4 \\*
		&\quad+ \frac{1}{2}\left((aC_i - b B_i)\check{s}_1 + (-a B_i - b C_i)\check{s}_2 + (-t_1^*C_i+ t_2^* B_i) \check{s}_3 +(t_1^*B_i+ t_2^* C_i)\check{s}_4\right)
	\end{align*}
	with
	$B_i = -A_{i2}^3 - A_{i1}^4,\, C_i = A_{i1}^3 - A_{i2}^4$
	and $a_i = \partial a/\partial u_i,\, b_i = \partial b/\partial u_i$ for $i=1,2$.
	The other two basis vectors are given by
	\begin{align*}
		F_j = \Psi_*\biggl(\frac{\partial}{\partial t_j} \biggr) = \check{s}_j,\qquad j=1,2. 
	\end{align*}
	
	By \ref{item:cayleycond}, $X_\psi$ is Cayley in $M^8$ if and only if $\eta$, defined as
	\begin{align*}
		\eta(u,v,w,y) &= u^\flat\wedge X(v,w,y)^\flat + v^\flat\wedge X(w,u,y)^\flat + w^\flat\wedge X(u,v,y)^\flat + y^\flat\wedge X(v,u,w)^\flat \\*
		&\quad + u \lrcorner X(v,w,y) \lrcorner \Phi + v \lrcorner X(w,u,y) \lrcorner \Phi + w \lrcorner X(u,v,y) \lrcorner \Phi + y \lrcorner X(v,u,w) \lrcorner \Phi
	\end{align*}
	for $u,v,w,y \in T_{s}M,\, s \in M$, vanishes on $X_\psi$. 
	Using the formulas 
	\begin{align*}
		\Phi &= u^4 \bar{e}^1\bar{e}^2\bar{\nu}^3\bar{\nu}^4 - u^2 v^2(\omega_1\sigma^1 + \omega_2 \sigma^2 + \omega_3\sigma^3) + v^4\check{s}^1\check{s}^2\check{s}^3\check{s}^4 \nonumber\\
		&= u^4 \bar{e}^1\bar{e}^2\bar{\nu}^3\bar{\nu}^4 
		- u^2v^2(\bar{e}^1\bar{e}^2 + \bar{\nu}^3\bar{\nu}^4)(\check{s}^1\check{s}^2 + \check{s}^3\check{s}^4) 
		- u^2v^2(\bar{e}^1\bar{\nu}^3 + \bar{\nu}^4\bar{e}^2)(\check{s}^1\check{s}^3 + \check{s}^4\check{s}^2)\nonumber\\
		&\quad - u^2v^2(\bar{e}^1\bar{\nu}^4 + \bar{e}^2\bar{\nu}^3)(\check{s}^1\check{s}^4 + \check{s}^2\check{s}^3) 
		+ v^4 \check{s}^1\check{s}^2\check{s}^3\check{s}^4
	\end{align*}
	and $X(u,v,w)^\flat = w \lrcorner v \lrcorner u \lrcorner \Phi$ (see \cite[\Sec A]{Merkel2024} for the complete computation), 
	we compute 
	\begin{align*}
		&E_1^\flat\wedge X(E_2,F_1,F_2)^\flat + E_2^\flat\wedge X(F_1,E_1,F_2)^\flat + F_1^\flat\wedge X(E_1,E_2,F_2)^\flat + F_2^\flat\wedge X(E_2,E_1,F_1)^\flat\\*
		&\quad = u^2 v^4\bigl((a_1+b_2) (-\bar{e}^1\check{s}^3 - \bar{e}^2\check{s}^4+ \bar{\nu}^3\check{s}^1+ \bar{\nu}^4\check{s}^2) 
		+ (a_2-b_1)(\bar{e}^1\check{s}^4 - \bar{e}^2\check{s}^3 + \bar{\nu}^3\check{s}^2 -\bar{\nu}^4\check{s}^1)\bigr) \\*
		%%%
		&\qquad+ \frac{1}{2}u^2v^4\Bigl(
		\bigl(t_1^*(-B_2+C_1)+ t_2^* (-B_1-C_2)\bigr)(\bar{e}^1\check{s}^3 + \bar{e}^2 \check{s}^4- \bar{\nu}^3\check{s}^1- \bar{\nu}^4\check{s}^2) \\*
		&\quad\qquad\qquad\quad+ \bigl(t_1^*(-B_1-C_2)+ t_2^* (B_2-C_1)\bigr) (\bar{e}^1\check{s}^4 - \bar{e}^2 \check{s}^3 + \bar{\nu}^3\check{s}^2 - \bar{\nu}^4\check{s}^1)\Bigr)
		\intertext{and}
		&E_1 \lrcorner X(E_2,F_1,F_2) \lrcorner \Phi + E_2 \lrcorner X(F_1,E_1,F_2) \lrcorner \Phi + F_1 \lrcorner X(E_1,E_2,F_2) \lrcorner \Phi + F_2 \lrcorner X(E_2,E_1,F_1) \lrcorner \Phi \\*
		%%%
		&\quad = 3u^2v^4\bigl((a_1+b_2)(\bar{e}^1\check{s}^3 +\bar{e}^2\check{s}^4 - \bar{\nu}^3 \check{s}^1 - \bar{\nu}^4\check{s}^2) + 
		(a_2-b_1)(-\bar{e}^1\check{s}^4+\bar{e}^2\check{s}^3 - \bar{\nu}^3 \check{s}^2 + \bar{\nu}^4\check{s}^1)\bigr) \\*	
		%%%
		&\qquad +\frac{3}{2}u^2v^4\Bigl(
		\bigl(t_1^*(B_2-C_1)+ t_2^* (B_1+C_2)\bigr)(\bar{e}^1\check{s}^3+\bar{e}^2\check{s}^4 - \bar{\nu}^3\check{s}^1 - \bar{\nu}^4\check{s}^2) \\*
		&\quad\qquad\qquad\quad+\bigl(t_1^*(B_1+C_2)+ t_2^* (-B_2+C_1)\bigr)(\bar{e}^1\check{s}^4 -\bar{e}^2\check{s}^3 + \bar{\nu}^3\check{s}^2 - \bar{\nu}^4\check{s}^1)\Bigr),
		\intertext{which add up to}
		&\eta(E_1, E_2, F_1, F_2) \\
		&\quad = 2u^2v^4\bigl((a_1+b_2)(\bar{e}^1\check{s}^3 +\bar{e}^2\check{s}^4 - \bar{\nu}^3 \check{s}^1 - \bar{\nu}^4\check{s}^2) + 
		(a_2-b_1)(-\bar{e}^1\check{s}^4+\bar{e}^2\check{s}^3 - \bar{\nu}^3 \check{s}^2 + \bar{\nu}^4\check{s}^1)\bigr) \\*	
		&\qquad +u^2v^4\Bigl(
		\bigl(t_1^*(B_2-C_1)+ t_2^* (B_1+C_2)\bigr)(\bar{e}^1\check{s}^3+\bar{e}^2\check{s}^4 - \bar{\nu}^3\check{s}^1 - \bar{\nu}^4\check{s}^2) \\*
		&\quad\qquad\qquad\quad+\bigl(t_1^*(B_1+C_2)+ t_2^* (-B_2+C_1)\bigr)(\bar{e}^1\check{s}^4 -\bar{e}^2\check{s}^3 + \bar{\nu}^3\check{s}^2 - \bar{\nu}^4\check{s}^1)\Bigr).
	\end{align*} 
	
	Given that $\xi^* \in (V_+)_{p^*}$ was arbitrary, all fibers of the restricted tangent bundle \smash{$T X_\psi \vert_{(X_\psi)_{p^*}}$} are Cayley if and only if this expression vanishes at $p^*$ for all $t^* \in \R^2$. Since $u,v > 0$, this is equivalent to the conditions
	\begin{enumerate}[label=(\Roman*)]
		\item $a_1 + b_2 = 0$ and $a_2 - b_1=0$ at $p^*$,\label{cond:Cayleyholomorphic}
		\item $B_2-C_1 = 0$ and $B_1+C_2=0$ at $p^*$.\label{cond:Cayleyminimal}
	\end{enumerate}
	Substituting the definitions of $B_i$ and $C_i$ yields
	\begin{align*}
		B_2-C_1 &= (-A_{22}^3 - A_{12}^4) - (A_{11}^3-A_{12}^4) = -(A_{11}^3 + A_{22}^3) = - \Tr A^3, \\
		B_1+C_2& = (-A_{12}^3 - A_{11}^4) + (A_{12}^3-A_{22}^4) = -(A_{11}^4 + A_{22}^4) = - \Tr A^4,
	\end{align*}
	which shows that \ref{cond:Cayleyminimal} is equivalent to $\Tr A \vert_{p^*} = 0$.
	
	As for \ref{cond:Cayleyholomorphic}, recall that $\psi$ is locally given by $\psi = as_3 + bs_4$ and that it is holomorphic if $\bar{\partial}_{V_-} \psi = (\nabla^{V_-})^{0,1} \psi =0$. Since $e_1 + i e_2$ locally trivializes $(TL)^{0,1}$, the latter is equivalent to {$\nabla^{V_-}_{e_1 + ie_2} \psi =0$}.
	We compute
	\begin{align*}
		\nabla^{V_-}_{e_1 + i e_2} \psi 
		&= \nabla^{V_-}_{e_1} \psi + J_-( \nabla^{V_-}_{e_2} \psi) \\*
		&= \pi_{V_-}(\nabla_{e_1} \psi ) + J_-\bigl(\pi_{V_-}(\nabla_{e_2} \psi)\bigr) \\*
		&= \pi_{V_-}(a \nabla_{e_1} s_3 + b \nabla_{e_1}s_4 + a_1 s_3 + b_1 s_4)\\*
		&\quad + J_-\bigl(\pi_{V_-}(a \nabla_{e_2} s_3 + b \nabla_{e_2}s_4 + a_2 s_3 + b_2 s_4)\bigr) \\
		&= a_1 s_3 + b_1 s_4 + J_-(a_2 s_3 + b_2 s_4) \\
		&= (a_1+b_2)s_3 + (-a_2+b_1)s_4,
	\end{align*}
	where the fourth equality follows from \autoref{lemma:nablaspinor} and the fact that $s_1$ and $s_2$ are orthogonal to $V_-$. 
	Consequently, \ref{cond:Cayleyholomorphic} holds if and only if $\bar{\partial}_{V_-} \psi $ vanishes at $p^*$.
	
	Since $p^* \in L$ was arbitrary, $ X_\psi$ is Cayley in $M^8$, which is equivalent to \smash{$T X_\psi \vert_{(X_\psi)_{p}}$} having Cayley fibers for all $p \in L$, if and only if $L$ is minimal in $X$ and $\psi$ is holomorphic.
\end{proof}

\begin{rem}
	In the same way, we obtain an analogous result for $\chi + V_-$ with $\chi \in \Gamma(V_+)$. 
\end{rem}

\section{Some explicit examples in the \texorpdfstring{$\Gtwo$}{G2} case}\label{sec:examples}

In this section, we use \autoref{thm:assNEU} to construct explicit examples of associative and coassociative submanifolds in the space of anti-self-dual 2-forms $\Lambda^2_-(T^*S^4)$ on the round sphere $S^4$. 
To do so, we examine two isometric immersions of 2-spheres into $S^4$: The equatorial sphere $S^2 = S^4 \cap (\R^3 \times \{(0,0)\}) \subset \R^5$ and the Veronese immersion of the sphere $S^2(\sqrt{3})$ of radius $\sqrt{3}$. 
As we will show below, they are negative and positive superminimal, respectively, and in particular minimal. 

In order to apply the first part of \autoref{thm:assNEU}, we need to study the holomorphic sections of the bundle \vphantom{$\int^{B^B}$}\smash{$F \stackrel{\text{loc}}{=} \{f^2, f^3\}$}\vphantom{$\int^B$} over $L^2 = S^2, S^2(\sqrt{3})$ with respect to the holomorphic structure $\bar{\partial}_F = (\nabla^F)^{0,1}$. 
Unlike in the proof, we may work with an oriented local orthonormal adapted frame $(e_1, e_2, \nu_3, \nu_4)$ which does not satisfy \eqref{eq:defnormalcoords}. Therefore, the condition for a (local) section $\sigma = af^2 + bf^3$ to be holomorphic may not be $a_1 = b_2$ and $a_2 = - b_1$. Instead, we have to redo the computation using a more general version of \autoref{lemma:nablaf}. 

Since $e_1, e_2, \nu_3, \nu_4$ are orthonormal, we find 
\begin{align*}
	\nabla_{e_j} f^1 &= ( \Gamma_{j4}^1 -\Gamma_{j3}^2 )f^2 + (-\Gamma_{j3}^1-\Gamma_{j4}^2)f^3, \\
	\nabla_{e_j} f^2 &= ( \Gamma_{j3}^2 -\Gamma_{j4}^1 )f^1 + (\Gamma_{j2}^1-\Gamma_{j4}^3)f^3, \\
	\nabla_{e_j} f^3 &= ( \Gamma_{j3}^1 +\Gamma_{j4}^2 )f^1 + (\Gamma_{j4}^3 -\Gamma_{j2}^1)f^2 
\end{align*}
for $j=1,2$, where $\Gamma_{jk}^l$ denotes the connection coefficients with respect to the given frame. 
Using this, we compute
\begin{align*}
	\nabla^F_{e_1 + i e_2} \sigma
	&= \nabla^F_{e_1} \sigma + J_F( \nabla^F_{e_2} \sigma) \\*
	&= \pi_F(\nabla_{e_1} \sigma ) + J_F\bigl(\pi_F(\nabla_{e_2} \sigma)\bigr) \\*
	&= \pi_F(a \nabla_{e_1} f^2 + b \nabla_{e_1}f^3 + a_1 f^2 + b_1 f^3)\\*
	&\quad+ J_F\bigl(\pi_F(a \nabla_{e_2} f^2 + b \nabla_{e_2}f^3 + a_2 f^2 + b_2 f^3)\bigr) \\
	&= a  (\Gamma_{12}^1-\Gamma_{14}^3)f^3 + b (\Gamma_{14}^3 -\Gamma_{12}^1)f^2  + a_1 f^2 + b_1 f^3 \\*
	&\quad + J_F\bigl(a  (\Gamma_{22}^1-\Gamma_{24}^3)f^3 + b (\Gamma_{24}^3 -\Gamma_{22}^1)f^2 + a_2 f^2 + b_2 f^3\bigr) \\*
	&= (a_1-b_2)f^2 + (a_2+b_1)f^3\\*
	&\quad + \bigl(- a  (\Gamma_{22}^1-\Gamma_{24}^3) + b (\Gamma_{14}^3 -\Gamma_{12}^1)\bigr) f^2 + \bigl(a  (\Gamma_{12}^1-\Gamma_{14}^3) + b (\Gamma_{24}^3 -\Gamma_{22}^1)\bigr)f^3,
\end{align*}
where $a_j = da(e_j)$ and $b_j = db(e_j)$ for $j=1,2$. Thus, the conditions for holomorphicity in this more general setting are 
\begin{align*}
	a_1-b_2 =  (\Gamma_{22}^1-\Gamma_{24}^3)\, a - (\Gamma_{14}^3 -\Gamma_{12}^1)\, b \quad \text{ and } \quad a_2+b_1 =  (\Gamma_{14}^3 -\Gamma_{12}^1)\, a + (\Gamma_{22}^1-\Gamma_{24}^3)\, b
\end{align*}
or, equivalently, 
\begin{align}\label{eq:updatedholocond}
	G_1 + i\, G_2  = (\Gamma_{22}^1-\Gamma_{24}^3)\, G + i\, (\Gamma_{14}^3 -\Gamma_{12}^1)\, G
\end{align}
for $G= a+ib$ and $G_j = dG(e_j)$, $j=1,2$. 

After solving this partial differential equation, it remains to examine whether the solutions holomorphically extend to global sections of $F$. 
A necessary condition for this is that $\lvert \sigma \rvert^2 = 2(a^2 + b^2) = 2 \lvert G \rvert^2$ is uniformly bounded because $L^2 = S^2, S^2(\sqrt{3})$ is compact. 
If there exist solutions of this kind, we cover $TS^4 \vert_{L}$ with oriented local orthonormal adapted frames, determine how the coefficients of $\sigma$ transform under change of frame, and examine whether the condition \eqref{eq:updatedholocond} remains fulfilled. 

We will show that, while the only globally defined holomorphic section for $L = S^2$ is the zero section, there do exist nontrivial holomorphic sections in the case of $S^2(\sqrt{3})$. 
The first part of \autoref{thm:assNEU} then yields examples of associative submanifolds in $\Lambda^2_-(T^*S^4)$ based on these sections. This shows that twisting $E$ only leads to new associative examples for the Veronese immersion, not the equatorial sphere.

As for the coassociative case, the second part of \autoref{thm:assNEU} applied to $L^2 = S^2$ gives that $\eta + F$ for $\eta \in \Gamma(E)$ and $E = \spn \{f^1 = \vol_{S^2} - *_{S^4} \vol_{S^2}\} \subset \Lambda^2_-(T^*S^4)\vert_{S^2}$ is coassociative in $\Lambda_-^2(T^*S^4)$ whenever $\eta$ is a constant multiple of $f^1$. 
The Veronese immersion, on the other hand, is positive instead of negative superminimal. However, since the antipodal map $x \mapsto -x$ on $S^4$ is an orientation-reversing isometry, it identifies positive superminimal submanifolds with negative superminimal ones and vice versa. Hence, the composition of this map and the Veronese immersion defines a negative superminimal submanifold of $S^4$, for which the given frame $(e_1, e_2, \nu_3,\nu_4)$ is negatively oriented. Consequently, the bundles corresponding to the modified immersion are given by $E^- = \spn \{e^1 \wedge e^2 + \nu^3 \wedge \nu^4\}$ and \smash{$F^- = (E^-)^\perp$}. According to the second part of \autoref{thm:assNEU}, $\eta^- + F^-$ for $\eta^- \in \Gamma(E^-)$ is coassociative in $\Lambda_-^2(T^*S^4)$ if and only if $\eta^-$ is a constant multiple of $e^1 \wedge e^2 + \nu^3 \wedge \nu^4$. 
Hence, we find new coassociative examples in both cases.  

In the remainder of this section, we show that the equatorial sphere and the Veronese immersion are indeed negative and positive superminimal, respectively, and discuss the corresponding holomorphic sections in more detail. 

\subsection{The equatorial sphere}

The equatorial embedding of $S^2 \subset \R^3$ into $S^4 \subset \R^5$ is given by 
\begin{align*}
	\iota: S^2 \rightarrow S^4, \quad (x,y,z) \mapsto (x,y,z,0,0).
\end{align*}
Since this is naturally totally geodesic, the second fundamental form vanishes identically, which in particular implies that $S^2$ is negative superminimal in $S^4$. 

Using stereographic projection from $N = (1,0,0) \in S^2$ onto $\R^2$, the immersion \smash{$\iota\vert_{S^2 \setminus \{N\}}$} can be written as
\begin{align*}
	\iota: \R^2 \rightarrow S^4, \quad (u_1, u_2) \mapsto \biggl(\frac{\lvert u \rvert^2 - 1}{\lvert u \rvert^2 + 1}, \frac{2 u_1}{\lvert u \rvert^2 + 1}, \frac{2 u_2}{\lvert u \rvert^2 + 1}, 0, 0\biggr). 
\end{align*}
This leads to the oriented local orthonormal adapted frame 
\begin{align*}
	e_1 &= \frac{\lvert u \rvert^2 + 1}{2}\frac{\partial \iota}{\partial u_1} 
	=  \biggl(\frac{2u_1}{\lvert u \rvert^2 + 1}, \frac{1 - u_1^2 + u_2^2}{\lvert u \rvert^2 + 1}, - \frac{2 u_1 u_2}{\lvert u \rvert^2 + 1}, 0, 0\biggr),  &&\nu_3 = (0,0,0,1,0), \\*
	e_2 &= \frac{\lvert u \rvert^2 + 1}{2}\frac{\partial \iota}{\partial u_2} 
	=  \biggl(\frac{2u_2}{\lvert u \rvert^2 + 1}, - \frac{2 u_1 u_2 }{\lvert u \rvert^2 + 1}, \frac{1 + u_1^2 - u_2^2}{\lvert u \rvert^2 + 1}, 0, 0\biggr),  &&\nu_4 = (0,0,0,0,-1),
\end{align*}
where we omit the dependence on $u$ for simplicity. 

As the Levi-Civita connection on $S^4$ is the tangential part of the directional derivative in $\R^5$, the connection coefficients with respect to the given frame can be computed via
\begin{align*}
	\Gamma_{jk}^l  = \frac{\lvert u \rvert^2 + 1}{2} \biggl\langle \frac{\partial e_k}{\partial u_j}, e_l \biggr\rangle \qquad \text{ for } j,k,l=1,2,
\end{align*}
and analogously for $k,l = 3,4$. Using $\Gamma_{jk}^l = - \Gamma_{jl}^k$ because the frame is orthonormal, we find that the only nonzero coefficients with $j \in \{1,2\}$ are
\begin{align*}
	\Gamma_{11}^2 = - \Gamma_{12}^1 = u_2 \quad \text{ and } \quad - \Gamma_{21}^2 = \Gamma_{22}^1 = u_1. 
\end{align*}

Since $G_j = dG(e_j) =  (\lvert u \rvert^2 + 1) \, \partial_{u_j} G / 2$ for $j=1,2$,
the condition \eqref{eq:updatedholocond} with respect to ($e_1, e_2, \nu_3, \nu_4$) becomes
\begin{align*}
	\frac{\lvert u \rvert^2 + 1}{2}(\partial_{u_1} G + i \,\partial_{u_2}G)  = (u_1+ i u_2) \,G 
	\iff  (z \bar{z} + 1)\, \partial_{\bar{z}} G = z G 
	\iff \partial_{\bar{z}} \frac{G}{z \bar{z} + 1} = 0
\end{align*}
for $z = u_1 + i u_2$. 
In other words, \eqref{eq:updatedholocond} is satisfied if and only if $G(z) = H(z)(z \bar{z} + 1)$ for a holomorphic function $H: \C \rightarrow \C$.
If $\lvert \sigma \rvert^2 = 2(a^2 + b^2) = 2 \lvert G \rvert^2$ is uniformly bounded, so is $H$. 
In that case, Liouville's theorem forces $H$ to be constant, hence $G$ must be of the form $G(z) = C( z \bar{z} + 1)$ for $C \in \C$. Since this function is only bounded for $C = 0$, the only globally defined holomorphic section of $F$ is the zero section. 

\subsection{The Veronese immersion of \texorpdfstring{$S^2(\sqrt{3})$}{}}

The Veronese immersion of $S^2(\sqrt{3}) \subset \R^3$ into $S^4 \subset \R^5$ is given by 
\begin{align*}
	\iota: S^2(\sqrt{3}) \rightarrow S^4, \quad (x,y,z) \mapsto \frac{1}{\sqrt{3}}  \biggl(x y, x z, y z, \frac{x^2 - y^2}{2}, \frac{x^2 + y^2 - 2 z^2}{2 \sqrt{3}} \biggr)
\end{align*}
(see, \eg \cite[\Ch 1, \Sec 4]{LawsonMinimalSubmfs80}, \cite[\Sec 6.2]{DajczerMarcos2009}).
(This also represents an isometric embedding of the real projective plane $\mathbb{RP}^2$ into $S^4$ by viewing $\mathbb{RP}^2$ as the quotient of $S^2(\sqrt{3})$ by the antipodal relation.)
We cover $S^2(\sqrt{3})$ with two types of spherical coordinates:
\begin{align*}
	\psi^{-1} : \psi(U) = (0, \pi) \times (0, 2\pi) &\rightarrow U = S^2(\sqrt{3}) \setminus \{ x \geq 0,\, y = 0\}, \\
	(\varphi, \theta) &\mapsto \sqrt{3}\, (\sin{\varphi} \cos{\theta}, \sin{\varphi} \sin{\theta},\cos{\varphi}); \\[0.5em]
	\hat{\psi}^{-1} : \hat{\psi}(\hat{U}) = (0, \pi) \times (- \tfrac{\pi}{2}, \tfrac{3 \pi}{2})&\rightarrow \hat{U} = S^2(\sqrt{3}) \setminus \{ x \leq 0,\, z = 0\},\\
	(\hat{\varphi}, \hat{\theta}) &\mapsto \sqrt{3}\, (\sin{\hat{\varphi}} \sin{\hat{\theta}},  \cos{\hat{\varphi}}, \sin{\hat{\varphi}} \cos{\hat{\theta}}).
\end{align*}

With respect to the coordinates $(\varphi, \theta)$, the immersion $\iota\vert_U$ can be written as
\begin{align*}
	\frac{2}{\sqrt{3}}\iota(\varphi, \theta) 
	= \sin^2\varphi\, Y_1(\theta) + \sin 2 \varphi\, Y_2(\theta) + \frac{1}{\sqrt{3}}(1-3\cos^2\varphi)\, \mathbf{e}_5,
\end{align*}
where 
\begin{align*}
	Y_1(\theta) = \sin2\theta\, \mathbf{e}_1 + \cos2\theta\, \mathbf{e}_4, \qquad Y_2(\theta) = \cos\theta\, \mathbf{e}_2 + \sin\theta\, \mathbf{e}_3
\end{align*}
only depend on $\theta$, and $\mathbf{e}_1, \dots, \mathbf{e}_5$ is the standard basis of $\R^5$. 
Then 
\begin{align*}
	Y_3(\theta) = \frac{1}{2} (Y_1)^\prime (\theta) = \cos2\theta\, \mathbf{e}_1 - \sin2\theta\, \mathbf{e}_4, \qquad Y_4(\theta) =  (Y_2)^\prime(\theta) = -\sin\theta\, \mathbf{e}_2 + \cos\theta\,\mathbf{e}_3
\end{align*}
allow us to write
\begin{align*}
	\frac{2}{\sqrt{3}} \frac{\partial \iota}{\partial \varphi}(\varphi, \theta) &= \sin2\varphi \, Y_1(\theta) +  2 \cos 2 \varphi\, Y_2(\theta) + \sqrt{3}\sin2\varphi\, \mathbf{e}_5,\\*
	\frac{2}{\sqrt{3}} \frac{\partial \iota}{\partial\theta}(\varphi, \theta) &= 2 \sin^2\varphi\, Y_3(\theta) + \sin 2 \varphi\, Y_4(\theta).
\end{align*}
Since $Y_1, \dots, Y_4, \mathbf{e}_5$ are orthonormal, constructing an adapted orthonormal frame now reduces to normalizing vectors and computing a cross product in $\R^3 \cong \spn \{Y_1(\theta), Y_2(\theta), \mathbf{e}_5\}$, yielding
\begin{align*}
	e_1 &= \frac{1}{\sqrt{3}}\frac{\partial \iota}{\partial\varphi} 
	=  \frac{1}{2} \sin2\varphi\, Y_1 +  \cos 2 \varphi\, Y_2 + \frac{\sqrt{3}}{2}\sin2\varphi\, \mathbf{e}_5, \\
	e_2 &= \frac{1}{\sqrt{3} \sin \varphi}\frac{\partial \iota}{\partial\theta} = \sin\varphi\, Y_3 + \cos\varphi\, Y_4,\\
	\nu_3 &= - \cos \varphi\, Y_3 + \sin \varphi\, Y_4, \\ 
	\nu_4 &= \frac{1}{2}\bigl((1+ \cos^2 \varphi)\, Y_1 - \sin 2 \varphi\, Y_2 - \sqrt{3} \sin^2 \varphi\, \mathbf{e}_5\bigr),
\end{align*}
where we omit the dependence on $(\varphi, \theta)$ for simplicity. 
A calculation verifies that this is indeed a positively oriented adapted orthonormal frame. 

Similarly to before, we have
\begin{align*}
	\Gamma_{1k}^l  = \frac{1}{\sqrt{3}} \biggl\langle \frac{\partial e_k}{\partial \varphi}, e_l \biggr\rangle, \qquad 
	\Gamma_{2k}^l =  \frac{1}{\sqrt{3} \sin \varphi} \biggl\langle \frac{\partial e_k}{\partial \theta}, e_l \biggr\rangle, \qquad \text{ for } k,l=1,2,
\end{align*}
and analogous formulas for $k,l = 3,4$. Using $(Y_1)^\prime = 2 Y_3,\, (Y_2)^\prime = Y_4,\, (Y_3)^\prime = - 2 Y_1$ and $ (Y_4)^\prime = -Y_2$, we find that the only nonzero coefficients $\Gamma_{jk}^l$ with $j \in \{1,2\}$ are
\begin{align*}
	\Gamma_{11}^4 = - \Gamma_{12}^3 = \Gamma_{13}^2 =- \Gamma_{14}^1 = -\Gamma_{21}^3 = - \Gamma_{22}^4 =\Gamma_{23}^1 = \Gamma_{24}^2 &=\frac{1}{\sqrt{3}}, \\
	2 \Gamma_{21}^2 = -2 \Gamma_{22}^1 = \Gamma_{23}^4 = - \Gamma_{24}^3 &= \frac{2\cot \varphi}{\sqrt{3}}.
\end{align*}
Consequently, the second fundamental form is given by 
\begin{align*}
	A^{3} = 
	\begin{pmatrix}
		\Gamma_{13}^1 & \Gamma_{13}^2 \\
		\Gamma_{23}^1 & \Gamma_{23}^2
	\end{pmatrix}
	= 
	\frac{1}{\sqrt{3}}
	\begin{pmatrix}
		0  & 1 \\
		1 & 0
	\end{pmatrix}, \qquad 
	A^{4} = 
	\begin{pmatrix}
		\Gamma_{14}^1 & \Gamma_{14}^2 \\
		\Gamma_{24}^1 & \Gamma_{24}^2
	\end{pmatrix}
	= 
	\frac{1}{\sqrt{3}}
	\begin{pmatrix}
		-1  &  0 \\
		0 & 1
	\end{pmatrix},
\end{align*}
and satisfies $A^4 = + J_T A^3$. In other words, $U \subset S^2(\sqrt{3})$ is positive superminimal. 

On the other hand, we can write $\iota\vert_{\hat{U}}$ as
\begin{align*}
	\iota(\hat{\varphi}, \hat{\theta}) = \sqrt{3} \sin \hat{\varphi} \cos \hat{\varphi}\, \hat{Y}_1(\hat{\theta}) + \sin^2 \hat{\varphi} \, \hat{Y}_2(\hat{\theta}) + \cos^2\hat{\varphi}\, \biggl(- \frac{\sqrt{3}}{2} \mathbf{e}_4 + \frac{1}{2} \mathbf{e}_5\biggr),
\end{align*}
where 
\begin{align*}
	\hat{Y}_1(\hat{\theta}) &= \sin \hat{\theta}\, \mathbf{e}_1 + \cos \hat{\theta}\, \mathbf{e}_3, \\ 
	\hat{Y}_2(\hat{\theta}) &= \sqrt{3} \sin \hat{\theta} \cos \hat{\theta}\, \mathbf{e}_2+ \frac{\sqrt{3}}{2} \sin^2 \hat{\theta}\, \mathbf{e}_4 + \frac{1}{2} (1-3 \cos^2 \hat{\theta})\, \mathbf{e}_5.
\end{align*}
In a similar manner as before, we find the oriented orthonormal frame
\begin{align*}
	\hat{e}_1 &= \frac{1}{\sqrt{3}}\frac{\partial \iota}{\partial\hat{\varphi}} =  \cos 2 \hat{\varphi}\, \hat{Y}_1 + \frac{1}{\sqrt{3}}\sin 2 \hat{\varphi}\, \hat{Y}_2 - \frac{1}{\sqrt{3}}\sin 2\hat{\varphi}\, \biggl(- \frac{\sqrt{3}}{2} \mathbf{e}_4 + \frac{1}{2} \mathbf{e}_5\biggr),\\
	\hat{e}_2 &= \frac{1}{\sqrt{3} \sin \hat{\varphi}}\frac{\partial \iota}{\partial\hat{\theta}} =  \cos \hat{\varphi}\, \hat{Y}_3 + \sin \hat{\varphi} \, \hat{Y}_4 , \\
	\hat{\nu}_3 & =  \sin \hat{\varphi}\, \hat{Y}_3  - \cos \hat{\varphi}\, \hat{Y}_4 , \\
	\hat{\nu}_4 & =  -  \sin\hat{\varphi} \cos \hat{\varphi}\, \hat{Y}_1 + \frac{1}{ \sqrt{3}}(1 + \cos^2 \hat{\varphi})\, \hat{Y}_2 + \frac{1}{\sqrt{3}} (1 + \sin^2 \hat{\varphi})\, \biggl(- \frac{\sqrt{3}}{2} \mathbf{e}_4 + \frac{1}{2} \mathbf{e}_5\biggr), 
\end{align*}
where 
\begin{align*}
	\hat{Y}_3(\hat{\theta}) &=  (\hat{Y}_1)^\prime (\hat{\theta}) = \cos \hat{\theta}\, \mathbf{e}_1 - \sin \hat{\theta}\, \mathbf{e}_3, \\*
	\hat{Y}_4(\hat{\theta}) &=  \frac{1}{\sqrt{3}}(\hat{Y}_2)^\prime(\hat{\theta}) =  \cos 2 \hat{\theta}\, \mathbf{e}_2+ \frac{1}{2} \sin 2 \hat{\theta} \, \mathbf{e}_4 + \frac{\sqrt{3}}{2} \sin 2 \hat{\theta} \, \mathbf{e}_5.
\end{align*}
As before, $\hat{Y}_1, \dots, \hat{Y}_4$ are orthonormal, but this time {$\langle \hat{Y}_j, - \frac{\sqrt{3}}{2} \mathbf{e}_4 + \frac{1}{2} \mathbf{e}_5 \rangle = - \frac{1}{2} \delta_{2j}$}. 
Using {$(\hat{Y}_1)^\prime = \hat{Y}_3$}, {$(\hat{Y}_2)^\prime = \sqrt{3} \hat{Y}_4,\, (\hat{Y}_3)^\prime = - \hat{Y}_1$} and {$(\hat{Y}_4)^\prime = - \frac{4}{\sqrt{3}} \hat{Y}_2 - \frac{2}{\sqrt{3}}(- \frac{\sqrt{3}}{2} \mathbf{e}_4 + \frac{1}{2} \mathbf{e}_5 )$}, we find that the connection coefficients $\hat{\Gamma}_{jk}^l$ with respect to $(\hat{e}_1, \hat{e}_2, \hat{\nu}_3, \hat{\nu}_4)$ are the same as $\Gamma_{jk}^l$ for $j \in \{1,2\}$, but with $\hat{\varphi}$ in place of $\varphi$. 
Thus, we also have {$\hat{A}^4 = + J_T \hat{A}^3$}, which proves that the Veronese immersion of \smash{$S^2(\sqrt{3})$} is positive superminimal in $S^4$. 

As $G_1 = dG(e_1) =  \partial_\varphi G / \sqrt{3}$ and \smash{$G_2 = dG(e_2) = \partial_\theta G / (\sqrt{3} \sin \varphi)$},
the condition \eqref{eq:updatedholocond} with respect to ($e_1, e_2, \nu_3, \nu_4$) is
\begin{align}\label{eq:holocondVeronese}
	\sin \varphi \, \partial_\varphi G + i\, \partial_{\theta} G = \cos \varphi \, G.
\end{align}
The method of characteristics yields the solution 
\begin{align*}
	G(\varphi, \theta) = \sin \varphi \cdot H\Biggl(\theta - i \ln \biggl(\tan \frac{\varphi}{2}\biggr)\Biggr)
\end{align*}
for any holomorphic function $H: \{z \in \C \mid \re\, z \in (0, 2 \pi)\} \rightarrow \C$.
The necessary condition $\lvert \sigma \rvert^2 = 2(a^2 + b^2) = 2 \lvert G \rvert^2 \leq M < \infty$ holds, for example, when $H$ is constant, in which case the section takes the form $\sigma = C  \sin \varphi\, f^2 +  D  \sin \varphi\, f^3$ for $C, D \in \R$.

Under the transformation 
$(
\sin{\varphi} \cos{\theta},
\sin{\varphi} \sin{\theta},
\cos{\varphi}
)=(
\sin{\hat{\varphi}} \sin{\hat{\theta}},
\cos{\hat{\varphi}},
\sin{\hat{\varphi}} \cos{\hat{\theta}}
)$, 
the vectors $Y_1, \dots, Y_4$ become
\begin{align*}
	\sin^2 \varphi \, Y_1 
	&= 2\sin{\hat{\varphi}} \cos{\hat{\varphi}} \sin{\hat{\theta}} \, \mathbf{e}_1 + (\sin^2{\hat{\varphi}} \sin^2{\hat{\theta}} - \cos^2{\hat{\varphi}})\, \mathbf{e}_4, \\
	\sin \varphi\,Y_2 &= \sin{\hat{\varphi}} \sin{\hat{\theta}}\, \mathbf{e}_2 + \cos{\hat{\varphi}} \, \mathbf{e}_3, \\
	\sin^2 \varphi \, Y_3 &=  (\sin^2{\hat{\varphi}} \sin^2{\hat{\theta}} - \cos^2{\hat{\varphi}})\, \mathbf{e}_1 - 2\sin{\hat{\varphi}} \cos{\hat{\varphi}} \sin{\hat{\theta}}\, \mathbf{e}_4, \\
	\sin \varphi\, Y_4 &=  -\cos{\hat{\varphi}}\, \mathbf{e}_2 + \sin{\hat{\varphi}} \sin{\hat{\theta}} \,\mathbf{e}_3,
\end{align*}
where {$\sin \varphi = \sqrt{1 - \sin^2 \hat{\varphi} \cos^2 \hat{\theta}}$}.
Substituting these expressions into 
\begin{align*}
	\sin \varphi\, e_1 &=  \cos \varphi\, (\sin^2\varphi\, Y_1) +  (\cos^2 \varphi - \sin^2 \varphi)(\sin \varphi\, Y_2) + \sqrt{3}\sin^2\varphi \cos\varphi\, \mathbf{e}_5 \\
	&=\sin \hat{\varphi} \cos \hat{\theta}\, (\sin^2\varphi\, Y_1) +  (2\sin^2 \hat{\varphi} \cos^2 \hat{\theta} - 1) (\sin \varphi\, Y_2) \\
	&\quad + \sqrt{3} \sin \hat{\varphi} \cos \hat{\theta}\, (1 - \sin^2 \hat{\varphi} \cos^2 \hat{\theta})\, \mathbf{e}_5, \\[0.5em]
	\sin \varphi\, e_2 &=( \sin^2\varphi\, Y_3) + \cos\varphi\, (\sin \varphi\, Y_4) = (\sin^2\varphi\, Y_3) + \sin{\hat{\varphi}} \cos{\hat{\theta}}\, (\sin \varphi\, Y_4),\\[0.5em]
	\sin^2 \varphi\, \nu_3 &= - \cos \varphi\,(\sin^2 \varphi\, Y_3) + \sin^2 \varphi\,(\sin \varphi\, Y_4) \\
	&= - \sin{\hat{\varphi}} \cos{\hat{\theta}}\, (\sin^2 \varphi\, Y_3) + (1 - \sin^2 \hat{\varphi} \cos^2 \hat{\theta})(\sin \varphi\, Y_4), \\[0.5em]
	\sin^2 \varphi\, \nu_4 &= \frac{1}{2}\bigl((1+ \cos^2 \varphi)(\sin^2 \varphi\, Y_1) - 2\sin^2 \varphi \cos \varphi\, (\sin\varphi\, Y_2) - \sqrt{3} \sin^4 \varphi\, \mathbf{e}_5\bigr)\\*
	&= \frac{1}{2}\bigl((1+ \sin^2{\hat{\varphi}} \cos^2{\hat{\theta}})(\sin^2 \varphi\, Y_1) - 2(1 - \sin^2 \hat{\varphi} \cos^2 \hat{\theta}) \sin{\hat{\varphi}} \cos{\hat{\theta}}\,(\sin\varphi\, Y_2)\\*
	&\qquad - \sqrt{3} (1 - \sin^2 \hat{\varphi} \cos^2 \hat{\theta})^2 \mathbf{e}_5\bigr),
\end{align*}
and projecting the latter onto the orthonormal frame vectors $\hat{e}_1, \hat{e}_2, \hat{\nu}_3, \hat{\nu}_4$ gives
\begin{align*}
	\sin \varphi\, e_1 &= - \cos \hat{\varphi} \cos \hat{\theta}\, \hat{e}_1 + \sin \hat{\theta}\, \hat{e}_2, \\
	\sin \varphi\, e_2 &= - \sin \hat{\theta}\, \hat{e}_1 - \cos \hat{\varphi} \cos \hat{\theta}\,  \hat{e}_2, \\
	\sin^2 \varphi\, \nu_3 &= \frac{1}{4}(-1+ 3 \cos 2\hat{\theta} + 2 \cos^2\hat{\theta}\cos 2 \hat{\varphi})\, \hat{\nu}_3 - \cos \hat{\varphi} \sin 2 \hat{\theta}\,  \hat{\nu}_4, \\
	\sin^2 \varphi\, \nu_4 &= \cos \hat{\varphi} \sin 2 \hat{\theta}\, \hat{\nu}_3 +  \frac{1}{4}(-1+ 3 \cos 2\hat{\theta} + 2 \cos^2\hat{\theta}\cos 2 \hat{\varphi})\,\hat{\nu}_4.
\end{align*}
Consequently, $f^2$ and $f^3$ transform as 
\begin{align*}
	\sin^3 \varphi\, f^2 
	&=  (\sin \varphi\, e^1) \wedge (\sin^2 \varphi\, \nu^3) - (\sin^2 \varphi\,\nu^4) \wedge (\sin \varphi\,e^2)  \\
	&=   \frac{1}{4}\cos \hat{\varphi} \cos \hat{\theta}\,(1- 3 \cos 2\hat{\theta} - 2 \cos^2\hat{\theta}\cos 2 \hat{\varphi} - 8 \sin^2  \hat{\theta})\, \hat{f}^2 \\
	&\quad +  \frac{1}{4}\sin \hat{\theta}\,(8\cos^2 \hat{\varphi}\cos^2\hat{\theta}  +1- 3 \cos 2\hat{\theta} - 2 \cos^2\hat{\theta}\cos 2 \hat{\varphi})\, \hat{f}^3 \\
	&= \sin^2 \varphi \, (- \cos \hat{\varphi} \cos \hat{\theta} \, \hat{f}^2 + \sin \hat{\theta} \, \hat{f}^3),\\[0.5em]
	\sin^3 \varphi\, f^3 
	&=  (\sin \varphi\, e^1) \wedge (\sin^2 \varphi\, \nu^4) - (\sin \varphi\,e^2) \wedge (\sin^2 \varphi\,\nu^3)  \\
	&=  \frac{1}{4} \sin  \hat{\theta}\, (- 8 \cos^2 \hat{\varphi} \cos^2 \hat{\theta} -1+ 3 \cos 2\hat{\theta} + 2 \cos^2\hat{\theta}\cos 2 \hat{\varphi})\, \hat{f}^2 \\
	&\quad + \frac{1}{4} \cos \hat{\varphi} \cos \hat{\theta}\,( 1- 3 \cos 2\hat{\theta} - 2 \cos^2\hat{\theta}\cos 2 \hat{\varphi}
	- 8 \sin^2  \hat{\theta}
	)\, \hat{f}^3 \\
	&= \sin^2 \varphi \, ( - \sin \hat{\theta} \, \hat{f}^2 - \cos \hat{\varphi} \cos \hat{\theta} \, \hat{f}^3),
\end{align*}
where {$\sin \varphi = \sqrt{1 - \sin^2 \hat{\varphi} \cos^2 \hat{\theta}}$}.

Applying this to the section $\sigma = C \sin \varphi\, f^2+  D \sin \varphi\, f^3 $ yields
\begin{align*}
	\sigma
	= \hat{a} \hat{f}^2 + \hat{b} \hat{f}^3
	= (- C \cos \hat{\varphi} \cos \hat{\theta} - D \sin \hat{\theta} ) \hat{f}^2 + (C \sin \hat{\theta} - D \cos \hat{\varphi} \cos \hat{\theta} ) \hat{f}^3,
\end{align*}
which shows that $\sigma$ is globally defined. 
Due to the identical structure of the connection coefficients,
the condition \eqref{eq:updatedholocond} with respect to $(\hat{e}_1, \hat{e}_2, \hat{\nu}_3, \hat{\nu}_4)$ is the same as \eqref{eq:holocondVeronese}, but with $(\hat{\varphi}, \hat{\theta})$ in place of $(\varphi, \theta)$. A brief calculation confirms that $\hat{G} = \hat{a} + i \hat{b}$ satisfies this condition. Thus, $\sigma = C \sin \varphi\, f^2+  D \sin \varphi\, f^3 $ is a globally defined, holomorphic section of $F$ for all $C, D \in \R$.

\section{Conclusion}

Our findings demonstrate that the constructions of calibrated submanifolds in Euclidean spaces in \cite{CaliA2} cannot be entirely extended to the manifolds $T^*S^n$, $\Lambda^2_-(T^*X)$ ($X^4 = S^4, \mathbb{CP}^2$) and $\negspinS(S^4)$ considered in \cite{CaliB}.
While the results for the two spaces of exceptional holonomy are in line with the previous findings, the construction in $T^*S^n$ does not provide any new examples because the Lagrangian condition already requires the $1$-form to vanish.
As in \cite{CaliA2}, the (co\nobreakdash-)associative and Cayley subbundles constructed in \cite{CaliB} allow deformations destroying the linear structure of the fiber, while the base space $L^2$ remains of the same type after twisting, namely minimal or negative superminimal.
This implies that the moduli space of calibrated submanifolds near a calibrated subbundle of this kind not only contains deformations of the base $L$ but also of the fiber. 
In contrast, the special Lagrangian bundle construction in $T^*S^n$ is much more rigid than in the case of $T^*\R^n$. 

It would be interesting to study whether there exist other types of deformations in the above three cases and if we can find similar results for other manifolds of special holonomy. 
	
	\printbibliography

	\noindent
Mathematisches Institut, Universität Münster, Einsteinstrasse 62, 48149 Münster, Germany\\
\emph{E-mail address}: \href{mailto:rmerkel@uni-muenster.de}{rmerkel@uni-muenster.de}

\end{document}